\documentclass[12pt,reqno,tbtags]{amsart}
\usepackage{amsmath}
\usepackage{float}
\usepackage{graphicx}
\usepackage{amsthm}
\usepackage{amssymb}
\usepackage{framed}
\usepackage{tikz}
\usetikzlibrary{arrows}

\usepackage{caption,sansmath}
\DeclareCaptionFont{sansmath}{\sansmath}
\newcommand{\E}{\mathbb E}

\newtheorem{theorem}{Theorem}[section]
\newtheorem{lem}[theorem]{Lemma}

\newtheorem{corollary}[theorem]{Corollary}

\newtheorem{conjecture}{Conjecture}[section]

\newtheorem{claim}[theorem]{Claim}

\newtheorem{remark}[theorem]{Remark}

\newcommand{\mc}[1]{\mathcal{#1}}
\newcommand{\mf}[1]{\mathfrak{#1}}
\newcommand{\bb}[1]{\mathbb{#1}}
\newcommand{\brm}[1]{\operatorname{#1}}

\newcommand{\eps}{\varepsilon}
\newcommand{\pl}[1]{\brm{#1}}
\begin{document}
\date{}

\author{S. Norin, L. Yepremyan}

\title{Tur\'{a}n number of generalized triangles}

\begin{abstract}
The family $\Sigma_r$ consists of all $r$-graphs with three edges $D_1,D_2,D_3$ such that $|D_1\cap D_2|=r-1$ and $D_1 \triangle D_2 \subseteq D_3$.  A \emph{generalized triangle}, $\mc{T}_r  \in \Sigma_r$ is an $r$-graph on $\{1,2,\ldots,2r-1\}$ with three edges $D_1, D_2, D_3$, such that
$D_1=\{1,2,\dots,r-1, r\}, D_2= \{1, 2, \dots, r-1, r+1 \}$ and $D_3 = \{r, r+1, \dots, 2r-1\}.$

Frankl and F\"{u}redi conjectured that for all $r\geq 4$, $\brm{ex}(n,\Sigma_r) = \brm{ex}(n,\mc{T}_r )$ for all sufficiently large $n$ and they also proved it for $r=3$. Later, Pikhurko showed that the conjecture holds for $r=4$. In this paper we determine $\brm{ex}(n,\mc{T}_5)$ and $\brm{ex}(n,\mc{T}_6)$ for sufficiently large $n$, proving the conjecture for $r=5,6$.

\end{abstract}
\maketitle
\begin{section}{Introduction}
In this paper we consider $r$-uniform hypergraphs, which we call $r$-graphs for brevity. We denote the vertex set of an $r$-graph $\mc{G}$ by $V(\mc{G})$ and the number of its vertices by $\brm{v}(\mc{G})$.
Let $\mathfrak{F}$ be a family of $r$-graphs. An $r$-graph $\mc{G}$ is $\mathfrak{F}$-free if it does not contain any member of $\mathfrak{F}$ as a subgraph. The Tur\'{a}n function $\brm{ex}(n,\mathfrak{F})$ is the maximum size of an $\mathfrak{F}$-free $r$-graph of order $n$:
\[\brm{ex}(n,\mathfrak{F}) = \max\left\{|\mc{G}| : \brm{v}(\mc{G})=n \textrm{ and } \mc{G} \textrm{ is } \mathfrak{F}-\textrm{free}\right\}.\]
When $\mathfrak{F}$ contains just one element, say $\mathfrak{F}:=\{\mc{F}\}$, we write $\brm{ex}(n, \mc{F}):= \brm{ex}(n,\mathfrak{F})$.
Let $\mathfrak{T}_r$ be the family of all $r$-graphs with three edges such
that one edge contains the symmetric difference of the other two, and let  $\Sigma_r$ be the family of all $r$-graphs with three edges $D_1,D_2,D_3$ such that $|D_1\cap D_2|=r-1$ and $D_1 \cap D_2 \subseteq D_3$.   A \emph{generalized triangle}, $\mc{T}_r \in \Sigma_r$ is an $r$-graph on $[2r-1]$ with three edges $D_1, D_2, D_3$, such that
$D_1=\{1,2,\dots,r-1, r\}, D_2= \{1, 2, \dots, r-1, r+1 \}$ and $D_3 = \{r, r+1, \dots, 2r-1\}.$ 

Note that $\mc{T}_r \in \Sigma_r$ and $\Sigma_r\subseteq \mathfrak{T}_r$. Note also that for $r=2,3$, $\Sigma_r = \mathfrak{T}_r$. As a generalization of Tur\'{a}n's theorem, in \cite{katona} Katona suggested to determine $\brm{ex}(n,\mathfrak{T}_3)$. This question was answered by Bollob\'{a}s in \cite{bollobas}. He showed that for any $n\geq 3$ the complete
balanced $3$-partite $3$-graph (that is, the sizes of any two parts differ at most by one) is the unique extremal graph. Hence, 
\[\brm{ex}(n,\mathfrak{T}_3) = \left\lfloor\frac{n}{3}\right\rfloor\times  \left\lfloor\frac{n+1}{3}\right\rfloor\times  \left\lfloor\frac{n+2}{3}\right\rfloor.\]
Bollob\'{a}s conjectured that the same result holds for all $r\geq 4$.  In \cite{sidorenko}  Sidorenko proved the conjecture for $r=4$, in fact he showed that $\brm{ex}(n,\mathfrak{T}_4)=\brm{ex}(n,\Sigma_4)$ and determined the latter. However, Shearer~\cite{shearer96} showed that Bollob\'{a}s conjecture fails for $ r \geq 10$.

But what can be said about the relation between $\brm{ex}(n,\mc{T}_r )$ and $\brm{ex}(n, \Sigma_r)$? In \cite{erdossimonovits},  Erd\H{o}s and Simonovits proved that for any fixed $r$
\[\brm{ex}(n,\mathcal{T}_r) - \brm{ex}(n,\Sigma_r) = o(n^r).  \]
Later, in \cite{franklfuredi}, Frankl and F\"{u}redi conjectured that $\brm{ex}(n,\mathcal{T}_r) = \brm{ex}(n,\Sigma_r)$ for $n$ sufficiently large:
\begin{conjecture}[P. Frankl, Z. F\"{u}redi, \cite{franklfuredi}]\label{conjecture} For every $r\geq 4$, there exists $n_0:= n_0(r)$ such that for all $n\geq n_0$ 
\[\brm{ex}(n,\mc{T}_r ) = \brm{ex}(n,\Sigma_r).\]
\end{conjecture}
In their previous work \cite{franklfuredi1}, Frankl and F\"{u}redi showed that Conjecture~\ref{conjecture}  holds for $r=3$ with very large $n_0$. Keevash and Mubayi \cite{keevashmubayi} presented a different proof of this result; they showed that one can take $n_0=33$.

Recently, the conjecture for $r=4$ was proved by Pikhurko~\cite{pikhurko}. We show that the conjecture holds for $r=5,6$.
\begin{theorem}
\label{thm:maintheorem}There exists $n_0$ such that for all $n \geq n_0$, $\brm{ex}(n, \mc{T}_r ) = \brm{ex}(n,\Sigma_r)$ for $r=5,6$. Moreover, extremal graphs are blowups of the unique $(11, 5, 4)$ and $(12, 6, 5)$ Steiner systems for $r=5$ and $r=6$, respectively.
\end{theorem}

The proof of Theorem~\ref{thm:maintheorem} uses several classical tools, which have been widely applied to Tur\'an type problems since the early days of the subject. (A recent survey by Keevash~\cite{Kee11} gives an excellent overview of these techniques.)  Let us  now present  a brief outline of the proof.

First, we use the density result of Frankl and F\"uredi~\cite{franklfuredi}, who have determined  $\brm{ex}(n, \mc{T}_r )/ \binom{n}{r}$ asymptotically for $r=5,6$. The result of~\cite{franklfuredi} relies on the application of the Lagrangian method, which goes back to the proof of Tur\'an theorem by Motzkin and Straus~\cite{MotStr65}. We further utilize the Lagrangian method; in Section~\ref{sec:notation} we develop a set  of tools for transferring results from the weighted (Lagrangian) to the unweighted setting and in Section~\ref{sec:lagrangian} we prove  a stability result which essentially relies on the continuity properties of the weighted setting.

Second, in Section~\ref{sec:symmetrization} we employ and streamline the symmetrization procedure, used earlier by Sidorenko \cite{sidorenko}, Pikhurko \cite{pikhurko} and others. We prove a generic theorem which allows one, under certain conditions, to derive a global stability result from its local version. This theorem has potential further applications, in particular, we use it in~\cite{lagrangian}.

Third, we derive Theorem~\ref{thm:maintheorem} from a stability argument, showing that sufficiently dense $\mc{T}_r$-free graphs are close (in the edit distance) to the blowups of the respective Steiner systems for $r=5,6$. The stability technique for establishing structural information in the extremal setting originated with  Erd\"os-Simonovits Stability theorem~\cite{erdossimonovits}. We illustrate some of our terminology and techniques by giving a proof of that theorem in Section~\ref{sec:example}.

Section~\ref{sec:local} contains the bulk of the technical work in the paper, establishing  Theorem~\ref{thm:maintheorem} for graphs which are close to the blowups of the respective Steiner systems. Here we also need to extend the classical arguments. In most of the known applications of the stability method as an intermediate step one shows that
if the graph under consideration has density close to the maximum then it can be transformed into a subgraph of the conjectured extremal configuration (a blowup of the respective Steiner system, in our case) by removal of a small fraction of vertices. As noted in~\cite{pikhurko}, this is false in our case, and an additional counting argument is required. 

Finally, in Section~\ref{sec:finale} we combine the results of the previous sections to finish the proof of Theorem~\ref{thm:maintheorem}.
\end{section}
\section{Notation and Preliminary Results}\label{sec:notation}
\subsection{Notation}
For an $r$-graph $\mc{F}$ and $v \in V(\mc{F})$, we denote by $ L_{\mc{F}}(v)$  the \emph{link of the vertex \(v\)}: $$ L_{\mc{F}}(v):=\{I \in (V(\mc{F}))^{(r-1)} \: | \: I \cup \{v\} \in \mc{F} \}.$$  More generally, for  \(I\subseteq V(\mc{F})\) the \emph{link $L_{\mc{F}}(I)$  of \(I\)} is defined as 
$$L_{\mc{F}}(I) := \{J\subseteq V(\mc{F}) \: | \:  J \cap I = \emptyset, I \cup J\in \mc{F}\}.$$
In the above mentioned notation, we will skip the index $\mc{F}$ whenever $\mc{F}$ is understood from the context.

We say that an $r$-graph $\mc{G}$ is obtained from an $r$-graph $\mc{F}$ by \emph{cloning a vertex $v$ to a set $W$} if $\mc{F} \subseteq \mc{G}$, $V(\mc{G}) \setminus V(\mc{F}) = W\setminus\{v\}$ and $ L_{\mc{G}}(w)= L_{\mc{F}}(v)$ for every $w \in W$. We say that $\mc{G}$ is \emph{a blowup of $\mc{F}$} if $\mc{G}$ is isomorphic to an $r$-graph obtained from $\mc{F}$ by repeatedly cloning and deleting vertices. We denote the set of all blowups of $\mc{F}$ by $\mf{B}(\mc{F})$.
We say that a family $\mf{F}$ of $r$-graphs is \emph{clonable} if every blowup of any $r$-graph in $\mf{F}$, also lies in $\mf{F}$. The Hypergraph Removal Lemma~\cite{Gow07,RodSko06} allows one to restrict many arguments related to Tur\'an-type problems to clonable families, and some of the more general results of this paper hold for all clonable families.

Let us introduce another class of hypergraph families, which are important for us. For a family of $r$-graphs $\mf{F}$, let 
$$m(\mf{F},n):=\max_{\substack{\mc{F} \in \mf{F} \\ \brm{v}({\mc{F}}) = n}} |\mc{F}|.$$
We say that $\mf{F}$ is \emph{smooth} if there exists $\lim_{n \to \infty}m(\mf{F},n)/n^r$. For a smooth family $\mf{F}$ we denote the above limit by $m(\mf{F})$. Our first lemma establishes a connection between clonable and smooth families.

\begin{lem}\label{lem:smooth}
Every clonable family is smooth.
\end{lem}
\begin{proof}Let $\mf{F}$ be a clonable family of $r$-graphs. Let
$$d:=\limsup_{n \to \infty} \frac{m(\mf{F},n)}{n^r}.$$ We need to show that for every $0<\eps<1$ there exists $N>0$ such that $ m(\mf{F},n)/n^r \geq d-\eps$ for every $n \geq N$. Let $\mc{F} \in \mf{F}$ be chosen so that $|\mc{F}| \geq (d -\delta)\brm{v}(\mc{F})^r$ for $\delta:=\eps/(d+1)$. Let $s:=\brm{v}(\mc{F})$. For a positive integer $k$, let $\mc{F}^{(k)}$ be an $r$-graph obtained by cloning every vertex of $\mc{F}$ to a set of size $k$. Then $\mc{F}^{(k)} \in \mf{F}$, $\brm{v}(\mc{F}^{(k)})=ks$ and $|\mc{F}^{(k)}|=k^r|\mc{F}| \geq  (d -\delta)(ks)^r$. Therefore, for  $n \geq (s-1)r/\delta$, we have 
\begin{align*}
\frac {m(\mf{F},n)}{n^r} &\geq (d -\delta)\left(\frac{s\lfloor n/s \rfloor}{n}\right)^r \geq (d -\delta)\left(1 -\frac{s-1}{n}\right)^r \\ &\geq (d -\delta)\left(1 -\frac{(s-1)r}{n}\right)\geq (d-\delta)(1-\delta) \geq d -\eps,
\end{align*}
as desired.
\end{proof}

\subsection{Stability}
\label{sec:stability}
In this subsection we formalize and extend the notion of stability, which is ubiqutious in the analysis of Tur\'an-type problems.

Let $\mf{F}$ and $\mf{H}$ be two families of $r$-graphs. The definitions in this subsection will be typically applied to situations when $\mf{F}$ is the family whose maximum density we are trying to determine, $\mf{H}$ is a substantially more structured subfamily of $\mf{F}$, and our goal is to show that $m(\mf{F},n)=m(\mf{H},n)$ for sufficiently large $n$.
We define \emph{the distance $d_{\mf{F}}(\mc{F})$ from an $r$-graph $\mc{F}$ to a family $\mf{F}$} as 
\[ d_{\mf{F}}(\mc{F}):=\min_{\substack{\mc{F'}\in\mf{F} \\ \brm{v}(\mc{F}) = \brm{v}(\mc{F}')}}{|\mc{F}\triangle\mc{F}'|}.\]

For $\eps, \alpha>0 $, we say that $\mf{F}$ is $(\mf{H}, \eps, \alpha)$-\emph{locally stable} if there exists $n_0 \in \mathbb{N}$ such that for all $\mc{F}\in\mf{F}$ with $\brm{v}(\mc{F}) =n \geq n_0$ and $d_{\mf{H}}(\mc{F})\leq \eps n^r$ we have
\begin{equation}\label{eq:localstability}
|\mc{F}|\leq m(\mf{H},n) - \alpha d_{\mf{H}}(\mc{F}).
\end{equation}
We say that $\mf{F}$ is $\mf{H}$-\emph{locally stable}  if $\mf{F}$ is $(\mf{H}, \eps, \alpha)$-locally stable for some choice of $\eps$ and $\alpha$. We say that $\mf{F}$ is $(\mf{H}, \alpha)$-\emph{stable} if it is $(\mf{H}, 1, \alpha)$-\emph{locally stable}, that is the inequality (\ref{eq:localstability}) holds for all  $\mc{F}\in\mf{F}$ with $\brm{v}(\mc{F}) =n \geq n_0$. We say that \(\mf{F}\) is \(\mf{H}\)-stable, if \(\mf{F}\) is \((\mf{H}, \alpha)\)-stable for some choice of \(\alpha\). 

\begin{remark}\label{rem:stability} The classical notion of stability differs from the one we introduced here. To parallel that notion, we could define $\mf{F}$ to be $\mf{H}$-stable if for every $\eps>0$ there exists $\delta>0$ such that for all $\mc{F}\in\mf{F}$ with $v(\mc{F}) =n$ and $|\mc{F}|\geq m(\mf{H},n) - \delta n^r$ one has $d_{\mf{H}}(\mc{F}) \leq \eps n^r$. Our notion of stability is stronger in two respects:
\begin{itemize}
\item It implies linear dependence between $\delta$ and $\eps$ in the above definition.
\item It is meaningful in the regime $d_{\mf{H}}(\mc{F}) = o(n^r)$, allowing us to compute Tur\'an numbers exactly. Note that if $\mf{F}$  is $\mf{H}$-stable using our definition then  $m(\mf{H},n) \geq m(\mf{F},n)$ for sufficiently large $n$.
\end{itemize}  
We refer to our notion of stability as simply ``stability" as opposed to, for example, ``sharp stability", for brevity.
\end{remark}

\subsection{Vertex local stability} We also introduce a weaker version of stability  (i.e. the requirements imposed on the family are stronger), however, in certain cases,  as we will see, stability (as defined in Section~\ref{sec:stability}) can be derived from this version.

Let $\mf{H}$ be a smooth family of $r$-graphs.
For $\eps, \alpha>0$, we say that a family $\mf{F}$ of $r$-graphs is  $(\mf{H}, \eps, \alpha)$-\emph{vertex locally stable} if there exists $n_0 \in \mathbb{N}$ such that for all $\mc{F}\in\mf{F}$ with $\brm{v}(\mc{F}) =n \geq n_0$,  $d_{\mf{H}}(\mc{F})\leq \eps n^{r}$, and 
$| L_{\mc{F}}(v)| \geq \left(rm(\mf{H}) - \eps \right) n^{r-1}$  for every $v \in V(\mc{F})$, we have
\[|\mc{F}|\leq m(\mf{H},n) - \alpha d_{\mf{H}}(\mc{F}).\]
We say that $\mf{F}$ is  $\mf{H}$-\emph{vertex locally stable} if $\mf{F}$ is  $(\mf{H}, \eps, \alpha)$-vertex locally stable for some $\eps,\alpha$. In some cases vertex local stability implies local stability, which informally means that when proving inequality~(\ref{eq:localstability}) for an $r$-graph $\mc{F}$, we can assume that all the vertices of $\mc{F}$ have large degree.
 
\subsection{Weighted hypergraphs and Lagrangians}

Let $\mc{F}$ be an $r$-graph.  Let $\mc{M}(\mc{F})$ denote the set of probability distributions on $V(\mc{F})$, that is, the set of functions $\mu: V(\mc{F}) \to [0,1]$ such that $\sum_{v \in V(\mc{F})}\mu(v)=1$. We call a pair $(\mc{F},\mu)$, where $\mu \in \mc{M}(\mc{F})$, a \emph{weighted  graph}.  Two weighted graphs $(\mc{F},\mu)$ and $(\mc{F}',\mu')$ are \emph{isomorphic} if there exists an isomorphism $\varphi: V(\mc{F}) \to V(\mc{F}')$ between $\mc{F}$ and $\mc{F}'$ such that $\mu'(\varphi(v))=\mu(v)$ for every $v \in V(\mc{F})$. As in the case of unweighted graphs, we generally do not distinguish between isomorphic weighted graphs. 

We define \emph{the density $\lambda(\mc{F},\mu)$ of a weighted graph  $(\mc{F},\mu)$}, by  $$\lambda(\mc{F},\mu):=\sum_{F \in \mc{F} }{\prod_{v \in F}\mu(v)}.$$ The \emph{Lagrangian $\lambda(\mc{F})$ of an $r$-graph $\mc{F}$} is defined by $$\lambda(\mc{F}):=\max_{\mu \in \mc{M}(\mc{F})}\lambda(\mc{F},\mu).$$ For a family of $r$-graphs $\mf{F}$, let $\lambda(\mf{F}):=\sup_{\mc{F} \in \mf{F}}\lambda(\mc{F})$.

If an $r$-graph $\mc{F}'$ is obtained from an $r$-graph $\mc{F}$ by cloning a vertex $u \in V(\mc{F})$ to a set $W$, $\mu \in \mc{M}(\mc{F})$, $\mu' \in \mc{M}(\mc{F'})$, then we say that $(\mc{F}', \mu')$ is \emph{a one vertex blowup of 
$(\mc{F}, \mu)$}, if $\mu(v)=\mu'(v)$ for all $v \in V(\mc{F}) \setminus \{u\}$ and $\mu(u)=\sum_{w \in W}\mu'(w)$. We say that  $(\mc{F}', \mu')$ is \emph{a blowup} of $(\mc{F}, \mu)$ if  $(\mc{F}', \mu')$ is isomorphic to a weighted $r$-graph which can be obtained from $(\mc{F}, \mu)$ by repeatedly taking one vertex blowups. We denote by $\mf{B}(\mc{F},\mu)$ the family of weighted graphs isomorphic to the blowups of $(\mc{F},\mu)$. 

\begin{remark}\label{rem:blowup}
An $r$-graph $\mc{F}'$ is a blowup of $\mc{F}$ with $V(\mc{F})=[n]$ if and only if there exists a partition $\{P_1,P_2,\ldots,P_n\} $ of $V(\mc{F}')$ such that $\{v_1,v_2,\ldots,v_r\} \in \mc{F'}$, $v_j \in P_{i_j}$ for $j \in [r]$ if and only if $\{i_1,i_2,\ldots,i_r\} \in \mc{F}$. When $\mc{F}$ is understood from the context we refer to $\mc{P}=\{P_1,P_2,\ldots,P_n\}$ as \emph{a blowup partition of $\mc{F'}$}.
If $\mc{F}$ \emph{covers pairs}, that is, for every \(u,v\in V(\mc{F})\), there exists some $F\in \mc{F}$ containing $u$ and $v$,  then
the blowup partition is unique up to the order of parts and its elements are the maximal independent sets in $\mc{F}$. 

Let us also note that a weighted $r$-graph $(\mc{F}', \mu')$ is a blowup of $(\mc{F}, \mu)$ if and only if there exists a partition as above with the additional property $\sum_{v \in P_i}\mu'(v)=\mu(i)$, for every $i \in [n]$. 
\end{remark}

Next we  define the distance between weighted graphs. If $\mc{F}_1,\mc{F}_2$ are two $r$-graphs such that $V(\mc{F}_1)=V(\mc{F}_2)$
and $\mu \in \mc{M}(\mc{F}_1)(=\mc{M}(\mc{F}_2))$, we define
$$d'(\mc{F}_1,\mc{F}_2,\mu):=\sum_{F \in \mc{F}_1 \triangle \mc{F}_2}\prod_{v\in F}{\mu(v)}.$$ 
 We define \emph{the distance between general weighted $r$-graphs $(\mc{F}_1, \mu_1)$ and $(\mc{F}_2, \mu_2)$}, as $$d((\mc{F}_1, \mu_1), (\mc{F}_2, \mu_2)):=\inf d'(\mc{F}'_1,\mc{F}'_2,\mu),$$
where the infimum is taken over all $r$-graphs $\mc{F}'_1,\mc{F}'_2,$ with  $V(\mc{F}'_1)=V(\mc{F}'_2)$ and  $\mu \in \mc{M}(\mc{F}_1')=\mc{M}(\mc{F}_2')$ satisfying $(\mc{F}'_i,\mu) \in \mf{B}(\mc{F}_i,\mu_i)$ for $i=1,2$. If $(\mc{F}, \mu)$ is a weighted $r$-graph  and $\mf{F}$ is a family of $r$-graphs we define \emph{the distance  from  $(\mc{F}, \mu)$ to $\mf{F}$} as $$d^w_{\mf{F}}(\mc{F}, \mu) :=\inf_{\mc{F'} \in \mf{F}, \mu' \in \mc{M}(\mc{F}')}d((\mc{F}, \mu), (\mc{F}', \mu')).$$ 
We write \(d_{\mf{F}}(\mc{F}, \mu)\) instead of \(d_{\mf{F}}^{w}(\mc{F}, \mu)\), except for the cases when we want to emphasize the difference between weighted and unweighted distance.

\begin{lem}\label{lem:distancevsweight} For any family \(\mf{H}\), if  \(\mc{F}\) is a graph  with $\brm{v}(\mc{F})=n$, and $\xi \in \mc{M}(\mc{F})$ then
$$d_{\mf{H}}(\mc{F}) \leq \frac{r! n}{n-r^2} \binom{n}{r}d^w_{\mf{H}}(\mc{F},\xi).$$
\end{lem}
\begin{proof} Choose an arbitrary $0<\eps<1$ and let  $d:=d^w_{\mf{H}}(\mc{F},\xi)$. Let $(\mc{B},\mu)$ be a blowup of $(\mc{F},\xi)$ such that there exists  $\mc{H} \in \mf{H}$ satisfying $d((\mc{B},\mu),(\mc{H},\mu)) \leq d+\eps$.

Let $\mc{P}=\{P_1,P_2,\ldots,P_n\}$ be a blowup partition of $V(\mc{B})$. Suppose $v_1,v_2,\ldots,v_r$ are chosen independently at random from $V(\mc{H})$ according to the distribution $\mu$. Let $A$ be the event that
that $\{v_1,v_2,\ldots,v_r\}$ is a \emph{transversal} of $\mc{P}$, that is, 
$|\{v_1,v_2,\ldots,v_r\} \cap P_j| \leq 1$ for every $P_j \in \mc{P}$.
We have $$\Pr[A]= \prod_{i=0}^{r-1}\left(1-\frac{i}{n}\right) \geq \left(1-\frac{r}{n}\right)^r \geq 1-\frac{r^2}{n}.$$
Thus, it follows that
\begin{equation}\label{eq:distancevsweight1}
\Pr\left[\{v_1,v_2,\ldots,v_r\} \in \mc{B} \triangle \mc{H}  \:|\: A\right] \leq \frac{r!(d+\eps)n}{n-r^2}.
\end{equation}

Now  consider $v_1,v_2,\ldots,v_n$ to be chosen independently at random  according to the distribution given by $\mu$, 
such that $v_i \in P_i$ for every $i \in [n]$. Let $\mc{H}'$ and $\mc{B}^{'}$  be the random subgraphs induced by $\{v_1,v_2,\ldots,v_n\}$, respectively, in $\mc{H}$ and $\mc{B}$. It follows from (\ref{eq:distancevsweight1}) and the linearity of expectation that \begin{equation}\label{eq:distancevsweight2}
\E[|\mc{B}' \triangle \mc{H}'|] \leq \frac{r!(d+\eps)n}{n-r^2} \binom{n}{r}.
\end{equation}
As $\mc{B}'$ is isomorphic to $\mc{F}$, the inequality (\ref{eq:distancevsweight2}) implies the lemma.
\end{proof}

\subsection{Weighted Stability}
\label{sec:weghtedstability}
In this section we introduce the notion of weighted stability and relate it to (unweighted) stability. Let \(\mf{F}, \mf{H}\) be two graph families. For $\eps, \alpha>0 $, we say that $\mf{F}$ is $(\mf{H}, \eps, \alpha)$-\emph{weight locally stable} if for every $\mc{F} \in \mf{F}, \mu \in \mc{M}(\mc{F})$ such that
$d_{\mf{H}}(\mc{F}, \mu)\leq \eps$, we have $$\lambda(\mc{F}, \mu) \leq \lambda(\mf{H})-\alpha d_{\mf{H}}(\mc{F}, \mu).$$
We say that $\mf{F}$ is $\mf{H}$-\emph{weight locally stable} if $\mf{F}$ is $(\mf{H}, \eps, \alpha)$-weight locally stable for some choice of $\eps$ and $\alpha$. 

We say that $\mf{F}$ is $(\mf{H}, \alpha)$-\emph{weight stable} if $\mf{F}$ is $(\mf{H}, 1, \alpha)$-\emph{weight locally stable}. We say that $\mf{F}$ is $\mf{H}$-\emph{weight stable} if $\mf{F}$ is $(\mf{H}, \alpha)$-weight stable for some choice of $\alpha$. 

Finally, for weighted graphs we would also consider the direct analogue of the classical notion of stability discussed in Remark~\ref{rem:stability}. We  say that $\mf{F}$ is $\mf{H}$-\emph{weakly weight stable} if for every $\eps>0$ there exists $\delta>0$ such that for every $\mc{F} \in \mf{F}$ and $\mu \in \mc{M}(\mc{F})$ if $\lambda(\mc{F},\mu) \geq \lambda(\mf{H})-\delta$, then  $d_{\mf{H}}(\mc{F}, \mu) \leq \eps$. The following lemma, establishes a connection between weighted and unweighted stability.

\begin{lem}\label{lem:localplusweight} 
Let \(\mf{H}\) be a clonable family. If the family $\mf{F}$ is $\mf{H}$-locally stable and $\mf{H}$-weight stable, then $\mf{F}$ is $\mf{H}$-stable.
\end{lem}
\begin{proof}
Let $\alpha,\eps>0$ be such that the family $\mf{F}$ is $(\mf{H},\eps,\alpha)$-locally stable and $(\mf{H},\alpha)$-weight stable. We will show that $\mf{F}$ is $(\mf{H},\alpha/2)$-stable, that is, for every $\mc{F} \in \mf{F}$ with $n:=\brm{v}(\mc{F})$ sufficiently large,
\begin{equation}\label{eq:localplusweight}
|\mc{F}| \leq m(\mf{H},n)-\frac{\alpha}{2} d_{\mf{H}}(\mc{F}).
\end{equation}

We can assume that $d_{\mf{H}}(\mc{F})>\eps n^r$, since otherwise (\ref{eq:localplusweight}) holds because $\mf{F}$ is $(\mf{H},\eps,\alpha)$-locally stable. 

By Lemma~\ref{lem:smooth} the family $\mf{H}$ is smooth. We choose  $n$ to be sufficiently large so that
$
1 - r^2/n \geq 1/2
$
and 
$$m(\mf{H},n) \geq \left(m(\mf{H})-\frac{\alpha\eps}{2}\right)n^r.$$ Using Lemma~\ref{lem:distancevsweight} and the fact that $\mf{F}$ is  $(\mf{H},\alpha)$-weight stable, we have
\begin{align*}\label{eq:localplusweight2} 
\frac{|\mc{F}|}{n^r}&=
\lambda(\mc{F},\xi_{\mc{F}})\leq m(\mf{H})-\alpha d^w_{\mf{H}}(\mc{F},\xi) \notag\\ &\leq \left(\frac{m(\mf{H},n)}{n^r}+\frac{\alpha\eps}{2}\right)-\alpha\left(1-\frac{r^2}{n}\right) \frac{d_{\mf{H}}(\mc{F})}{r!\binom{n}{r}}\notag \\
&\leq \left(\frac{m(\mf{H},n)}{n^r}+\frac{\alpha\eps}{2}\right)-\frac{\alpha d_{\mf{H}}(\mc{F})}{2n^r}\notag \\&=\frac{(m(\mf{H},n) - \alpha/2 d_{\mf{H}}(\mc{F})) + \alpha/2(\eps n^r - d_{\mf{H}}(\mc{F}))}{n^r} \\&\leq \frac{m(\mf{H},n) - \alpha d_{\mf{H}}(\mc{F})/2}{n^r},
\end{align*}
implying (\ref{eq:localplusweight}).
\end{proof}

When both families \(\mf{F}\) and \(\mf{H}\) are clonable, weight local stability implies local stability, as follows.  For an $r$-graph $\mc{F}$, let $\xi_{\mc{F}} \in \mc{M}(\mc{F})$ denote the uniform distribution on $V(\mc{F})$, that is, $\xi_{\mc{F}}(v)=1/\brm{v}(F)$ for every $v \in V(\mc{F})$. We will omit the index and write $\xi$ instead of $\xi_{\mc{F}}$ when $\mc{F}$ is understood from the context. 
Note that $\lambda(\mc{F},\xi)=|\mc{F}|/(\brm{v}(\mc{F}))^r$. In the other direction, let $(\mc{F},\mu)$ be a weighted graph, choose $k$ integer such that $\mu(v)k$ is an integer for every $v \in V(\mc{F})$. Let $\mc{F}'$ be an $r$-graph obtained by cloning  $v \in V(\mc{F})$ to a set of size  $\mu(v)k$. Then, clearly, $\brm{v}(\mc{F}')=k$ and $|\mc{F}'|=\lambda(\mc{F},\mu)k^r$.  This second observation routinely implies the following lemma.

\begin{lem}\label{lem:weightedbasics} For every weighted $r$-graph $(\mc{F}, \mu)$ there exists a sequence $\{\mc{F}_n\}$ of blowups of $\mc{F}$, such that 
\begin{itemize}
\item $\brm{v}(\mc{F}_n) \to_{n \to \infty} \infty$
\item $\lim_{n \to \infty}\frac{|\mc{F}_n|}{\brm{v}(\mc{F}_n)^r}=\lambda(\mc{F},\mu)$
\item $\lim_{n \to \infty}\frac{d_{\mf{H}}(\mc{F}_n)}{\brm{v}(\mc{F}_n)^r}=d_{\mf{H}}(\mc{F},\mu)$ for every clonable family $\mf{H}$.
\end{itemize}
\end{lem}

Lemma~\ref{lem:weightedbasics} immediately implies the following.

\begin{corollary}\label{lem:localtoweighted} 
Let $\mf{F}, \mf{H}$ be two clonable families. If $\mf{F}$ is $\mf{H}$-locally stable then $\mf{F}$ is $\mf{H}$-weight locally stable.
\end{corollary}

\begin{section}{Local Stability From Vertex Local Stability}
\label{sec:genstability}

The main result of this section is the following important tool used in the proof of Theorem~\ref{thm:general}

\begin{theorem}
\label{thm:narrowlocaltolocal}
Let $\mf{F},\mf{H}$ be families of $r$-graphs such that $\mf{H}$ is clonable. If $\mf{F}$ is $\mf{H}$-vertex locally stable, then $\mf{F}$ is $\mf{H}$-locally stable.
\end{theorem}

In the proof of Theorem~\ref{thm:narrowlocaltolocal} we use the following two auxiliary lemmas.

\begin{lem}\label{degreelem}Let $\mf{F}$  be a clonable family of $r$-graphs. Then for every $\eps>0$ there exist $\delta>0$ and $n_0\in\mathbb{N}$ satisfying the following. For every $\mc{F}\in \mf{F}$ with $\brm{v}(\mc{F})=n \geq n_0$ and $|\mc{F}|\geq \left(m(\mf{F}) - \delta\right) n^r $ there exists $X\subseteq V(\mc{F})$ such that $|X|\geq (1-\eps)n$ and
\[\left| | L_{\mc{F}}(v)| - rm(\mf{F}) n^{r-1}\right| \leq \eps n^{r-1}\]
for every $v\in X$.
\end{lem}
\begin{proof}  Clearly, it is enough to prove the lemma for sufficiently small $\eps$. Thus we assume without loss of generality that $\max\{\eps,\eps^2r^2m(\mf{F})\}<1$. We show that $\delta:=(\eps^6 -\eps^8r^2m(\mf{F}))/(1+r+r^2)$ satisfies the lemma for sufficiently large $n_0$. Let $X \subseteq V(\mc{F})$ be the set of all $v \in V(\mc{F})$ satisfying
 \[\left| | L_{\mc{F}}(v)| - rm(\mf{F}) n^{r-1}\right| \leq \eps n^{r-1}.\] To prove that \(|X|\geq (1-\eps)n\), we first show the following claim.
 \begin{claim}  \[| L_{\mc{F}}(v)| \leq (rm(\mf{F}) + \eps^2) n^{r-1}\]
for every $v \in  V(\mc{F})$.
\end{claim}
\begin{proof}
Suppose for a contradiction that $$| L_{\mc{F}}(v)| > (rm(\mf{F}) + \eps^2) n^{r-1}$$ for some $v \in  V(\mc{F})$.
Let $n':=\lceil(1+\eps^4)n\rceil$ and let $\mc{F}'$ be obtained from $\mc{F}$ by cloning $v$ into a set of size $\lceil\eps^4 n\rceil+1$. We have $\mc{F}'\in \mf{F}$, as $\mf{F}$ is clonable. For sufficiently large $n$, we have
\begin{align}m(\mf{F},n') &\leq (m(\mf{F}) + \delta)n'^r \leq  (m(\mf{F}) + \delta)(1+\eps^4 r+ \eps^8r^2) n^r. \label{eq:Fprimeupperbound}
\end{align}
On the other hand,
\begin{align}m(\mf{F}, n') \geq |\mc{F}'| &> |\mc{F}| + \eps^4 n (rm(\mf{F}) + \eps^2 ) n^{r-1}\notag \\
&\geq (m(\mf{F}) - \delta) n^r + \eps^4 (rm(\mf{F}) + \eps^2 ) n^{r}\label{Fprimelowerbound}.
\end{align}
But now  (\ref{eq:Fprimeupperbound}) and (\ref{Fprimelowerbound}) together imply that
$$\eps^6 - \delta  < \delta(1+\eps^4 r + \eps^8r^2)+\eps^8r^2m(\mf{F}),$$
which contradicts to our choice of $\delta$. Thus, the claim holds.
\end{proof}
By the preceding claim we have that
 \[| L_{\mc{F}}(v)| < \left(rm(\mf{F})  - \eps\right) n^{r-1}\]
for all $v \in V(\mc{F}) \setminus X$. Now suppose for a contradicton that $|X|< (1-\eps)n$. Then
 \begin{align*}|\mc{F}| &= \frac{1}{r} \left( \sum_{v \in V(\mc{F}) \setminus X}{| L_{\mc{F}}(v)|}+ \sum_{v\in X}{| L_{\mc{F}}(v)|}\right) \\
 &< \frac{1}{r} \left( (n-|X|)  \left(rm(\mf{F})  - \eps\right) + |X| \left(rm(\mf{F})  + \eps^2\right) \right)n^{r-1}\\
 &=m(\mf{F})n^r + \frac{\eps}{r}\left((1+\eps)|X| -  n\right) n^{r-1} \\&<m(\mf{F})n^r -\frac{\eps^3}{r}n^r\\
 &\leq (m(\mf{F}) -\delta) n^r,
 \end{align*}
 a contradiction.
\end{proof}
\begin{lem}\label{lem:upperbound} Let $\mf{F}$  be a clonable family of $r$-graphs. Then for every $\eps >0$ there exist $n_0\in\mathbb{N}$ such that for all  $n_2 \geq n_1 \geq n_0$, we have 
\[m(\mf{F}, n_2)  \geq m(\mf{F}, n_1) + (n_2-n_1) (rm(\mf{F}) - \eps)n_1^{r-1}\]
\end{lem}
\begin{proof}
Consider $\mc{F}_1\in \mf{F}$ with $v(\mc{F}_1)= n_1$ such that $|\mc{F}_1| = m(\mf{F}, n_1).$ For large enough $n_1$ we have
\begin{align*}m(\mf{F},n_1) &\geq \left(m(\mf{F}) - \frac{\eps}{r}\right)n_1^r.
\end{align*}
By averaging, there exists $v \in V(\mc{F}_1)$ such that  \begin{equation}\label{highdegvertex}| L_{\mc{F}_1}(v)|  \geq \left(r m(\mf{F}) - \eps\right) n_1^{r-1}.
\end{equation} Let $\mc{F}_2$ be obtained from $\mc{F}_1$ by cloning $v$ to a set of size $n_2-n_1+1$. As $\mc{F}_2\in \mf{F}$, we have
\begin{align*}m(\mf{F}, n_2) \geq |\mc{F}_2| &\geq |\mc{F}_1| + (n_2-n_1)\left(rm(\mf{F}) - \eps\right) n_1^{r-1} \\
&= m(\mf{F}, n_1)+ (n_2-n_1)\left(rm(\mf{F}) - \eps\right) n_1^{r-1},
\end{align*}
as desired.
\end{proof}
\begin{proof}[Proof of Theorem~\ref{thm:narrowlocaltolocal}:] 
Let $\eps, \alpha$ be such that $\mf{F}$ is $(\mf{F}',\eps,\alpha)$-vertex locally stable.
We choose constants \(\eps', \eps''\) such that $0< \eps' \ll \eps'' \ll \eps$ so that the inequalities throughout the proof are satisfied. Let $\alpha':=\min\{\alpha, 2\eps''r^2(1 - m(\mf{H})) \}$. We will show that $\mf{F}$ is $(\mf{H},\eps',\alpha')$-locally stable.

Consider $\mc{F}\in \mf{F}$ with $V(\mc{F})=[n]$ and $d_{\mf{H}}(\mc{F}) \leq \eps' n^r$. We assume that 
\[|\mc{F}| \geq m(\mf{H}, n)- \eps' n^r, \]
since otherwise the result follows, as $\alpha'<1$.  Let $\mc{H}\in \mf{H}$ be such that 
$|\mc{F}\triangle \mc{H}| = d_{\mf{H}}(\mc{F})$. For large enough $n$, we have $|\mc{H}| \geq (m(\mf{H}) - \eps')n^r$.
By Lemma~\ref{degreelem} applied to $\mc{H}$ with $\eps=\eps''$, there exists $X\subseteq [n]$ with $|X| \geq (1-\eps'')n$ such that for each $v\in X$,
\begin{equation}
\label{degreecond}
\left|| L_{\mc{H}}(v)| - rm(\mf{H}) n^{r-1}\right| \leq \eps '' n^{r-1}.
\end{equation}
Consider the set \[J = \{v\in V(\mc{F}) : | L_{\mc{F}}(v)| < (rm(\mf{H}) - (2r^2+1)\eps'')n^{r-1}\}.\]
We will show that \(J\) has relatively small size. From the definition of $J$ and \(X\), it follows that for each \(v\in J\cap X\), we have
\(| L_{\mc{F}}(v)\triangle L_{\mc{H}}(v)|\geq \eps'' n^{r-1}.\)
Thus,
\[|J\cap X|\eps'' n^{r-1} \leq \sum_{v\in V(\mc{F})}{| L_{\mc{F}}(v) \triangle  L_{\mc{H}}(v)|} = r|\mc{F}\triangle \mc{H}| \leq \eps' r n^r,\]
and therefore, \(|J|\leq | J \cap X| + |J\setminus X| \leq (\frac{\eps'r}{\eps''}  + \eps'')n \leq 2\eps''n.\)
Let $\mc{F}':=\mc{F}|_{V(\mc{F})\setminus J}$, $\mc{H}':=\mc{H}|_{V(\mc{F})\setminus J}$ and $n':=n-|J|$. We have
\begin{equation}
\label{distanceOfF1}d_{\mf{H}}(\mc{F}')\leq |\mc{F}' \triangle \mc{H}'| \leq |\mc{F} \triangle \mc{H}| \leq \eps' n^r \leq \eps n'^r.
\end{equation}
Also, for every $v\in V(\mc{F})\setminus J$, we have
\begin{align}
| L_{\mc{F}'}(v)|\geq | L_{\mc{F}}(v)| - |J| n^{r-2}  &\geq \left(rm(\mf{H}) - 2r\eps''-2\eps''\right)n^{r-1} \notag \\
&\geq(rm(\mf{H}) - \eps)n'^{r-1} \label{degreeF2}.
\end{align}
Since  $\mf{F}$ is $(\mf{H}, \eps, \alpha)$-vertex locally stable, (\ref{distanceOfF1}) and (\ref{degreeF2}) imply that
\begin{equation}
\label{boundonF2}|\mc{F}'|\leq m(\mf{H}, n') - \alpha d_{\mf{H}}(\mc{F}').
\end{equation}
Let $\mc{H}''\in \mf{H}$ be such that $|\mc{H}''\triangle\mc{F}'| = d_{\mf{H}}(\mc{F}')$. Let $\mc{H}_0$ be obtained from $\mc{H}''$ by blowing up a vertex in $V(\mc{F})\setminus J$ to a set of size $n -n'+1$. We have
\begin{align}
|\mc{F}\triangle \mc{H}_0| 
&\leq  |\mc{F}'\triangle \mc{H}''|  + |J|n^{r-1}\label{upperboundonFB}.
\end{align}
By Lemma~\ref{lem:upperbound}, for sufficiently large $n$, we have 
\begin{align}m(\mf{H}, n) &\geq m(\mf{H}, n') + (n-n')\left(rm(\mf{H}) -\frac{\eps''}{1-2r\eps''}\right)n'^{r-1}\notag \\
&\geq m(\mf{H}, n') + |J|\left(rm(\mf{H}) - \frac{\eps''}{1-2r\eps''}\right)(1-2r\eps'')n^{r-1}. \label{maxFprimeOnN:maxFprimeOnN2}
\end{align}
Now we are ready to put all the obtained inequalities together to show that $\mf{F}$ is $(\mf{H}, \eps', \alpha')-$ locally stable.
\begin{align*}
|\mc{F}|&\leq |\mc{F}'| + |J|(rm(\mf{H}) - (2r^2+1)\eps'')n^{r-1} \\
&\stackrel{(\ref{boundonF2})}{\leq}\ m(\mf{H}, n') - \alpha d_{\mf{H}}(\mc{F}') + |J|(rm(\mf{H}) - (2r^2+1)\eps'')n^{r-1} \\
&\stackrel{(\ref{maxFprimeOnN:maxFprimeOnN2})}{\leq}
 m(\mf{H}, n) - |J| \left(rm(\mf{H}) -\frac{\eps''}{1-2r\eps''}\right)(1-2r\eps'')n^{r-1} \\&\qquad \qquad \;\;\;- \alpha |\mc{F}'\triangle\mc{H}''| + |J|(rm(\mf{H}) - (2r^2+1)\eps'')n^{r-1} \\
&= m(\mf{H}, n)  - \alpha |\mc{F}'\triangle \mc{H}''| - 2\eps''r^2(1- m(\mf{H}))|J|n^{r-1}  \\
&\leq m(\mf{H}, n) - \alpha'|\mathcal{F}'\triangle \mc{H}''| - \alpha' |J| n^{r-1} \\
&\stackrel{(\ref{upperboundonFB})}{\leq} m(\mf{H}, n) - \alpha'|\mathcal{F}\triangle \mc{H}_0| \\
&\leq m(\mf{H}, n)  - \alpha' d_{\mf{H}}(\mc{F}),
\end{align*}
as desired.
\end{proof}
\end{section}

\section{Weak stability from lagrangians}\label{sec:lagrangian}
In this section we prove that, under certain restrictions, every sufficiently dense graph in a family is close to some graph maximizing the lagrangian in that family. The arguments we use in this and the next section are continuous in nature.

We say that an $r$-graph $\mc{F}$ is \emph{thin} if for every  $(r-1)$-tuple $I\subseteq V(\mc{F})$, there exists at most one edge containing \(I\). In other words, $\mc{F}$ is thin if and only if it is $\mc{D}_r$-free, where $\mc{D}_r$ is an $r$-graph with two edges $D_1$ and $D_2$ such that $|D_1 \cap D_2|=r-1$. Note that every $(m,r,r-1)$ Steiner system is thin. We say that the family  $\mf{F}$ is \emph{thin} if every $\mc{F} \in \mf{F}$ is thin. In the applications of the next result the family \(\mf{F}^*\) will consist of the \(r\)-graphs which cover pairs. In particular, we do not assume that \(\mf{F}^*\) is clonable.

 \begin{theorem}\label{thm:compactness} If the family $\mf{F}^*$ is thin and the family $$\mf{F}^{**}=\{\mc{F}^*|_{\pl{supp}(\mu)} \: | \: \mc{F}^* \in \mf{F}^{*} , \: \lambda(\mc{F}^*,\mu)= \lambda(\mf{F}^{*})\; {for\; some} \; \mu \in \mc{M}(\mc{F}^*)\}.$$ \
is not empty, then $\mf{F}^*$ is $\mf{F}^{**}$-weakly weight stable.
\end{theorem}

\begin{proof} We will consider infinite $r$-graphs in the proof of this theorem.
Let $\mf{F}_{\bb{N}}$ denote the family of $r$-graphs such that $V(\mc{F})=\bb{N}$ for every $\mc{F} \in \mf{F}_{\bb{N}}$ and every finite subgraph $\mc{H}$ of a graph in $\mf{F}_{\bb{N}}$ is obtained from a subgraph of a graph in $\mf{F}^*$ by adding isolated vertices. Clearly, $\mf{F}_{\bb{N}}$ is thin.  We enhance $\mf{F}_{\bb{N}}$ with a metric $\varsigma$ defined as follows. For $\mc{F},\mc{F}' \in \mf{F}_{\bb{N}}$, let $\varsigma(\mc{F},\mc{F}'):=1/2^k$, where $k$ is the minimum integer such that $\mc{F}|_{[k]} \neq \mc{F}'|_{[k]}$. Note that $(\mf{F}_{\bb{N}},\varsigma)$ is compact.

Let $$\mc{M}(\bb{N}):=\{\mu: \bb{N} \to \bb{R}_+ \: | \: \mu(1) \geq  \mu(2) \geq \mu(3) \geq \ldots, \; \sum_{i=1}^{\infty}\mu(i) \leq 1\}.$$ 
It is not hard to verify that $\mc{M}(\bb{N})$ is compact with $L^1$ norm $\|\cdot\|_1$. Let $\mf{X}$ be the product of
$(\mf{F}_{\bb{N}},\varsigma)$ and $(\mc{M}(\bb{N}),\|\cdot\|_1)$. 

Note that every pair  $(\mc{F}, \mu)$ with $\mc{F} \in \mf{F}^*, \mu \in \mc{M}(\mc{F})$ naturally corresponds to an 
element of $\mf{X}$, as we can assume that $V(\mc{F})=[v(\mc{F})]$ and $\mu(i) \geq \mu(j)$ for all $i \leq j$, $i,j \in V(\mc{F})$.   

For $(\mc{F},\mu) \in \mf{X}$, define $\lambda(\mc{F},\mu):=\sum_{F \in \mc{F}} \mu(F).$ 
\begin{claim}
\label{lambdacontonX}
$\lambda$ is continuous on $\mf{X}$. 
\end{claim}
\begin{proof}
It is easy to see that $$|\lambda(\mc{F},\mu) - \lambda(\mc{F},\mu')| \leq \|\mu-\mu'\|_{1}$$
for every $\mc{F} \in \mf{F}_{\bb{N}}$ and all $\mu,\mu' \in \mc{M}(\bb{N})$. Thus, it suffices to show that 
for all $\mc{F},\mc{F}' \in\mf{F}_{\bb{N}}$ and every $\eps >0$ there exists $N\in\mathbb{N}$ such that if $\mc{F}'|_{[N]}=\mc{F}|_{[N]}$ then
$|\lambda(\mc{F},\mu) - \lambda(\mc{F}',\mu)| \leq \eps$ for every
$\mu \in \mc{M}(\bb{N})$. 
We show that  $N:=\lceil \frac{1}{\eps(r-1)!} \rceil$ satisfies the above. Let  $\mc{H}:=\mc{F}'|_{[N]}=\mc{F}|_{[N]}$. 
It suffices to show that $\lambda(\mc{F},\mu) \leq \lambda(\mc{H},\mu)+\eps$. We have 
\begin{align*}
 \lambda(\mc{F},\mu) - \lambda&(\mc{H},\mu) = \sum_{F \in \mc{F}, F \not \subseteq [N]} \prod_{i \in F}\mu(i) \\ 
&\leq \mu(N+1) \sum_{I \subseteq \bb{N}^{(r-1)}} \prod_{i \in I}\mu(i) \\
&\leq \mu(N+1)\frac{1}{(r-1)!}\left(\sum_{i \in N}\mu (i)\right)^{r-1} \leq \frac{1}{N(r-1)!} \leq \eps,
\end{align*}
as desired. Note that in the second inequality above we use the fact that \(\mc{F}\) is thin.
\end{proof}
It follows from the above claim that
 \begin{equation}
\label{eq:lambdamax}
\lambda(\mf{F}^{*})=\max\limits_{(\mc{F},\mu) \in \mf{X}}\lambda(\mc{F},\mu),
\end{equation}
as  every \((\mc{F},\mu)\in \mf{X}\) is a limit of a sequence of weighted graphs in $\mf{F}^*$. Let $$\mf{X}^{**}=\{ (\mc{F},\mu) \in \mf{X} \: | \: \mc{F}|_{\brm{supp}(\mu)}\in \mf{F}^{**}\},
$$
That is, $\mf{X}^{**}$ is a set of weighted graphs in $\mf{X}$ with finite support, coinciding with some graph in $\mf{F}^{**}$ on its support.

\begin{claim}\label{claim:weak2} If $\lambda(\mc{F},\mu)=\lambda(\mf{F}^{*})$ for some \((\mc{F},\mu)\in \mf{X}\), then $(\mc{F},\mu) \in \mf{X}^{**}$. 
\end{claim}
\begin{proof} Suppose for a contradiction that there exists some $(\mc{F},\mu) \in \mf{X}\setminus\mf{X}^{**}$ such that $\lambda(\mc{F},\mu)=\lambda(\mf{F}^{*})$. By definition of  \(\mf{F}^{**}\), it follows that $\brm{supp}(\mu)$ must be infinite, and hence, $\brm{supp}(\mu) = \mathbb{N}$, since $\mu$ is non-decreasing. As $\lambda(\mc{F}, \nu)$ considered as a function of $\nu$ is maximized at $\nu=\mu$ we have \[\frac{ \partial{\lambda(\mc{F}, \nu)}}{ \partial \nu(i)}\Big|_{\nu=\mu} = r\lambda(\mf{F}^{*}),\]
for every \(i\in \mathbb{N}\). Thus, we have
\begin{equation}\label{eq:lagrangianderivative}
\sum_{\substack{J \in \bb{N}^{(r-1)}, |J|=r-1 \\ J \cup \{i\} \in \mc{F}}} \prod_{j \in  J} \mu(j)=r\lambda(\mf{F}^{*}) \end{equation}
for every $i \in \bb{N}$. To show that (\ref{eq:lagrangianderivative}) cannot hold we employ an argument similar to the one used in the proof of the previous claim.
Choose an integer $N$ such that $N> \frac{1}{r(r-2)!\lambda(\mf{F}^{*})}$, and let $i$ be such that $|F \cap [N]| \leq r-2$ for every $F \in \mc{F}$ with $i \in F$. 
Then \begin{align*}
 \sum_{\substack{J \in \bb{N}^{(r-1)}, |J|=r-1 \\ J \cup \{i\} \in \mc{F}}} \prod_{j \in  J} \mu(j) &\leq  \mu(N+1)\sum_{K \in \bb{N}^{(r-2)}, |K|=r-2} \prod_{j \in K}\mu(j)\\ & \leq \frac{1}{N(r-2)!} < r\lambda(\mf{F}^{*}).
 \end{align*} This contradiction finishes the proof of the claim. \end{proof}

Now we are ready to finish the proof. We will show that for every \(\eps>0\) there exists \(\delta>0\) such that for every \(\mc{F}\in \mf{F}^*\)  and \(\mu \in \mc{M}(\mc{F})\), if 
\(\lambda(\mc{F}^*, \mu) \geq \lambda(\mf{F}^{*}) - \delta\), then \(d_{\mf{F}^{**}}(\mc{F}^*, \mu) \leq \eps\). (Clearly $\lambda(\mf{F}^*)$=$\lambda(\mf{F}^{**})$ so the above implies the theorem.)  Abusing  notation slightly we consider pairs $(\mc{F},\mu)$ as above as  elements of $\mf{X}.$ 

From continuity of $\lambda$ and Claim~\ref{claim:weak2} it follows that for every  \(\eps>0\) there exists \(\delta>0\) such that for every $(\mc{F},\mu) \in \mf{X}$ satisfying   \(\lambda(\mc{F}^*, \mu) \geq \lambda(\mf{F}^{*}) - \delta\)
there exists  \((\mc{F}^{**}, \mu^{**})\in \mf{X}^{**}\) such that  \(\mc{F}|_{[n]} = \mc{F}^{**}|_{[n]}\) for all \(n\leq \frac{2}{\eps}(r-1)! + 1\).

Following the argument in Claim~\ref{lambdacontonX},  let \(\mc{H}: =  \mc{F}|_{[N]} (= \mc{F}^{**}|_{[N]}\)), for \(N:= \lceil \frac{2}{\eps}(r-1)!\rceil\). As in Claim~\ref{lambdacontonX} we have
$$\lambda(\mc{F},\mu^*) - \lambda(\mc{H},\mu^*) \leq \frac{1}{N(r-1)!},$$
$$\lambda(\mc{F}^{**},\mu^*) - \lambda(\mc{H},\mu^*) \leq \frac{1}{N(r-1)!}.$$
Finally, we have
\begin{align*}d_{\mf{F}^{**}}(\mc{F}^*, \mu^*) &\leq d((\mc{F}^*, \mu^*), (\mc{H}, \mu^*)) +  d((\mc{H}, \mu^*), (\mc{F}^{**}, \mu^*)) \\
&\leq (\lambda(\mc{F}^*,\mu^*) - \lambda(\mc{H},\mu^*)) +  (\lambda(\mc{F}^{**},\mu^*) - \lambda(\mc{H},\mu^*)) \\
&\leq \frac{2}{N(r-1)!} \leq \eps,
\end{align*} 
as desired.
\end{proof}

\begin{section}{Stability from local stability}\label{sec:symmetrization}

Our next result can be considered as a generalization of the symmetrization argument of Sidorenko~\cite{sidorenko}, which was subsequently modified and employed by Pikhurko~\cite{pikhurko} and Hefetz and Keevash~\cite{HefKee13}. It can serve as a general tool to obtain global stability from local stability for clonable families. However, note that  although our main result, Theorem~\ref{thm:general}, uses this tool, it is not a direct application, since the family in our interests, \(\pl{Forb}(\mc{T}_r)\), is not clonable.

\begin{theorem}\label{thm:symmetrization}
Let $\mf{F},\mf{H}$ be clonable families of  $r$-graphs.  Let $\mf{F^*}$ consist of all $r$-graphs in $\mf{F}$ that cover pairs. If $\mf{F^*}$ is $\mf{H}$-weakly weight stable and $\mf{F}$ is $\mf{H}$-locally  stable then $\mf{F}$ is $\mf{H}$-stable.  
\end{theorem}
\begin{proof}
By Lemma~\ref{lem:localplusweight}, it suffices to show that $\mf{F}$ is $\mf{H}$-weight stable. By Corollary~\ref{lem:localtoweighted}  the family $\mf{F}$ is $\mf{H}$-weight  locally stable.
Let $\eps, \alpha>0$ be such that $\mf{F^*}$ is $(\mf{H},\alpha)$-weakly weight stable and $\mf{F}$ is $(\mf{H},\eps,\alpha)$-locally weight stable. Define $\delta:=\alpha\eps/2$.
We will prove that for every $\mc{F} \in \mf{F}$ and $\mu \in \mc{M}(\mc{F})$ such that
\begin{equation}
\label{eq:stability0}
 \lambda(\mc{F},\mu) \geq \lambda(\mf{H})-\delta,
\end{equation}
we have 
\begin{equation}
\label{eq:stability1}
d_{\mf{H}}(\mc{F},\mu) \leq \eps.
\end{equation}
Note that this statement implies that $\mf{F}$ is  $(\mf{H},\delta)$-weight stable as  $\mf{F}$ is $(\mf{H},\eps,\alpha)$-locally weight stable and $\delta \leq \alpha$.

The proof is by induction on  $\brm{v}(\mc{F})$. The  base of induction is trivial. For the induction step we assume that
$\mc{F} \not \in \mf{F}^*$, as otherwise (\ref{eq:stability1}) holds. Indeed, if $\mc{F} \in \mf{F}^*$, we have
$$d_{\mf{H}}(\mc{F},\mu) \leq \frac{\lambda(\mf{H})-\lambda(\mc{F},\mu)}{\alpha} \leq \frac{\delta}{\alpha} \leq \eps,$$ as $\mf{F^*}$ is $(\mf{H},\alpha)$-weight stable and $\delta \leq \alpha\eps$.

Thus, \(\mc{F}\in \mf{F}^*\) and there exist $v_1,v_2 \in V(\mc{F})$, such that  $\{v_1,v_2\} \not \subseteq F$ for every $F \in \mc{F}$. We assume that \(\mu(v_1)\neq 0\) and \(\mu(v_2)\neq 0\), since otherwise the conclusion follows from the induction hypothesis. We will consider a family  of probability distributions on $V(\mc{F})$ defined as follows. For $t \in [0,1]$, let $\mu_t \in \mc{M}(\mc{F})$ be defined by
$\mu_t(v) = \mu(v)$ for all $v \in V(\mc{F})\setminus\{v_1,v_2\}$,  $\mu_t(v_1)=t(\mu(v_1)+\mu(v_2))$, and $\mu_t(v_2)=(1-t)(\mu(v_1)+\mu(v_2))$. Note that $\mu=\mu_x$, for $x:=\mu(v_1)/(\mu(v_1)+\mu(v_2)$).  As \(\mu(v_1)\neq 0\) and \(\mu(v_2)\neq 0\), it follows that $x \not \in \{0,1\}$.

Note that $(\mc{F},\mu_0)$ and $(\mc{F},\mu_1)$ can be considered as weighted $r$-graphs on $\brm{v}(\mc{F})-1$ vertices and, therefore, the induction hypothesis is apllicable to them. Moreover,
\begin{equation}\label{eq:stability2}
\lambda(\mc{F},\mu)=x\lambda(\mc{F},\mu_0)+(1-x)\lambda(\mc{F},\mu_1).
\end{equation}
If $\lambda(\mf{F},\mu_i) < \lambda(\mf{H}) - \delta$  for $i=1,2$, then by (\ref{eq:stability2}), $\lambda(\mc{F},\mu) < \lambda(\mf{H})-\delta$, in contradiction with (\ref{eq:stability0}). Thus, without loss of generality, we  assume that $\lambda(\mf{F},\mu_0) \geq \lambda(\mf{H}) - \delta$. By the induction hypothesis we have $d_{\mf{H}}(\mc{F},\mu_0) \leq \eps$. 

Now suppose for a contradiction that $d_{\mf{H}}(\mc{F},\mu) > \eps$. As $d_{\mf{H}}(\mc{F},\mu_t)$ is a continuous function of $t$, there exists  $y \in [0,x]$ such that $d_{\mf{H}}(\mc{F},\mu_y)=\eps$. Since
$\mf{F}$ is $(\mf{H},\eps,\alpha)$-locally weight stable, we have
 \begin{equation}\label{eq:stability3}
\lambda(\mc{F},\mu_y) \leq \lambda(\mf{H}) - \alpha\eps.
\end{equation}
On the other hand, 
\begin{align}\label{eq:stability4}
\lambda&(\mc{F},\mu_y) = \frac{x-y}{x}\lambda(\mc{F},\mu_0)+\frac{y}{x}\lambda(\mc{F},\mu_x)\notag \\ &\geq \frac{x-y}{x}(\lambda(\mf{H}) - \delta) + \frac{y}{x}(\lambda(\mf{H}) - \delta) = \lambda(\mf{H}) - \delta>
 \lambda(\mf{H}) - \alpha\eps,
\end{align}
as $\delta< \alpha\eps$. The contradiction between inequalities (\ref{eq:stability3}) and (\ref{eq:stability4}) concludes the proof.
\end{proof}

\section{Erd\H{o}s-Simonovits Stability Theorem via local and weighted stability.}\label{sec:example}

In this subsection we give a sample application of the techniques we developed thus far. We give a proof of the classical Erd\H{o}s-Simonovits Stability Theorem~\cite{Sim68}, which can be stated in the language of this paper as follows.

\begin{theorem}[Erd\H{o}s-Simonovits Stability Theorem~\cite{Sim68}]
Let $t \geq 2$ be a fixed positive integer, and let $K_t$ denote the complete graph on $t$ vertices. Then $\brm{Forb}(K_t)$ is $\mf{B}(K_{t-1})$-stable.
\end{theorem}
 
\noindent \emph{Proof.} Let $\mf{F}:=\brm{Forb}(K_t)$ and $\mf{H}:=\mf{B}(K_{t-1})$. 

\begin{claim}\label{claim:erdosstoneaux}\(\mf{F}\) is \(\mf{H}\)-vertex locally stable.
\end{claim}

Our theorem follows from this claim. Indeed, by Theorem~\ref{thm:narrowlocaltolocal}, Claim~\ref{claim:erdosstoneaux} implies that $\mf{F}$ is $\mf{H}$-locally stable. Theorem~\ref{thm:symmetrization}  in turn implies that \(\mf{F}\) is \(\mf{H}\)-stable, as the family
$\mf{F}^*$ in the statement of Theorem~\ref{thm:symmetrization} is the family of cliques on at most $(t-1)$ vertices, and is, trivially, $\mf{H}$-weakly weight stable.
Thus it remains to prove the claim. 

\begin{proof}[Proof of Claim~\ref{claim:erdosstoneaux}] We will show that $\mf{F}$ is $(\mf{F}',\eps,1)$-vertex locally stable, that is, there exist $\eps>0$, \(n_0\in \mathbb{N}\) such that if  $\mc{F} \in \mf{F}$ satisfies $\brm{v}(\mc{F})=n \geq n_0$, $d_{\mf{H}}(\mc{F}) \leq \eps n^2$ and 
\begin{equation}\label{eq:degree}
| L_{\mc{F}}(v)| \geq \left(\frac{t-2}{t-1} -\eps\right)n,
\end{equation}
for every $v \in V(\mc{F})$, then $|\mc{F}| \leq m(\mf{H},n)-d_{\mf{H}}(\mc{F})$. In fact, we prove a stronger statement. We show that if the above conditions hold then there exists $\mc{H}_0 \in \mf{H}$ such that $\mc{F} \subseteq \mc{H}_0$, that is, $\mc{F}$ is $(t-1)$-partite. 

\begin{remark}An even stronger result was proved by Andr\'{a}sfai, Erd\H{o}s and S\'{o}s~\cite{AndErdSos74}. They show that the condition $d_{\mf{H}}(\mc{F}) \leq \eps n^2$ is unnecessary, and (\ref{eq:degree}) suffices to deduce that $\mc{F}$ is $(t-1)$-partite for $\eps < \frac{1}{(3t-4)(t-1)}$. We, however, include the proof which exploits the bound on the distance from $\mc{F}$ to $\mf{H}$ to demonstrate the methods used in the proof of Theorem~\ref{thm:general}.
\end{remark}

Let $0 \ll \eps \ll \gamma \ll 1/t$ be chosen to satisfy the inequalities appearing further in the proof  and let \(n\) be sufficiently large. Given $\mc{F}$ as above, let $\mc{H} \in \mf{H}$ be such that $V(\mc{H})=V(\mc{F})$ and $|\mc{F} \triangle \mc{H}| = d_{\mf{H}}(\mc{F})$. Since,  $d_{\mf{H}}(\mc{F}) \leq \eps n^2$, we have 
\begin{equation}\label{eq:ES2}
|\mc{H}| \geq |\mc{F}|- \eps n^2 \geq \left(\frac{t-2}{t-1} -3\eps\right)\frac{n^2}{2}.
\end{equation}

Let $\mc{P}=\{P_1,P_2,\ldots,P_{t-1}\}$ be the blowup partition of $V(\mc{H})$. It is easy to see that (\ref{eq:ES2}) implies that $$\left||P_i| - \frac{n}{t-1}\right| \leq \gamma n,$$ for all $i \in [t-1]$ with an appropriate choice of $\eps \ll\gamma$.

Next we show that the neighborhood of every vertex in $\mc{F}$ is ``close" to the neighborhood of some vertex in $\mc{H}$. The corresponding part of the proof of Theorem~\ref{thm:general}, Lemma~\ref{theorem2}, is longer and more technical then the argument below, yet the main ideas are very similar.   

For $v \in V(\mc{F})$, let $I(v)=\{ i \: | \: |N(v) \cap P_i| \geq \gamma n \}$, where $N(v)$ denotes the neighborhood of $v$. Then (\ref{eq:degree}) implies that $|I(v)| \geq t-2$ for every $v \in V(\mc{F})$.
Suppose that $|I(v)|=t-1$, and choose $Q_i \subseteq N(v) \cap P_i$  so that $|Q_i|=\gamma n$ for $i \in [t-1]$. For simplicity, we assume that $\gamma n$ is an integer. Let $Q = \cup_{i\in [t-1]}Q_i \subseteq N(v)$. 
Then $\mc{F}|_{Q}$ is $K_{t-1}$-free and, therefore, Tur\'{a}n's theorem implies that \begin{equation}\label{eq:ES3}
|\mc{F}|_{Q}| \leq \frac{(t-3)((t-1)\gamma n)^2}{2 (t-2)}
\end{equation}
On the other hand, $\mc{H}|_{Q}$ is $K_t$-free, thus, 
\begin{equation}\label{eq:ES4}
|\mc{H}|_{Q}| \leq \frac{(t-2)((t-1)\gamma n)^2}{2 (t-1)}.
\end{equation}
Combining (\ref{eq:ES2}) and (\ref{eq:ES3}), we deduce that
\begin{align*}
|\mc{F} \triangle \mc{H}| &\geq|\mc{F}|_Q \triangle \mc{H}|_Q | \\ &\geq  \left(
\frac{t-2}{t-1} - \frac{t-3}{t-2}\right) \frac{((t-1)\gamma n)^2}{2} > \eps n^2.\end{align*}
This contradiction implies that $|I(v)|=  t-2$ for all $v \in V(\mc{F})$.

Finally, we construct a partition $\mc{P}'=\{P_1',P_2',\ldots,P_{t-1}'\}$ of $V(\mc{F})$ so that $\mc{F} \subseteq \mc{F}''$, where $\mc{F}''$ is a blowup of $K_{t-1}$ with the blowup partition $\mc{P'}$. Define $P_i':= \{v \in V(\mc{F}) \: | \: i \not \in I(v)\}$ for $i \in [t-1]$. 
Note that (\ref{eq:degree}) and the bounds on the size of $P_j$ imply that $$|N(v) \cap P_j| \geq n/(t-1) - (t-1)\gamma n$$ for every $v \in P_i$, $i \neq j$. It follows that, if  $v,v' \in P_i$, then $\{v,v'\} \not \in \mc{F}$. (Otherwise, $\mc{F}|_{N(v) \cap N(v')}$ is $K_{t-2}$-free and $|N(v) \cap N(v') \cap P_j| \geq n/(t-1) - (2t-1)\gamma n$ for every $j \in [t-1] \setminus \{i\}$. This leads to a contradiction using an argument completely analogous to the one used in the preceding paragraph.) Thus, $\mc{F} \subseteq \mc{F}''$, as desired. 
\end{proof}

\end{section}

\begin{section}{Local stability of Forb($\mc{T}_r$)}\label{sec:local}
Recall that an $(m,r,r-1)$ \emph{Steiner system} is an $r$-graph on $m$ vertices such that every $(r-1)$-tuple is contained in a unique $r$-edge. Let $\mc{S}$ be an $(m,r,r-1)$ Steiner system, it is easy to see that $|\mc{S}|  =\frac{{m \choose r-1}}{r}$ and  $| L_{\mc{S}}(v)|=\frac{{m-1 \choose r-2}}{r-1}$ for every $v \in V(\mc{S})$.  We frequently use the following notation for related densities:
\begin{align*}
\pl{e}(m,r)&:=\frac{{m \choose r-1}}{rm^r},\\
\pl{d}(m,r)&:=\frac{{m-1 \choose r-2}}{(r-1)m^{r-1}}.
\end{align*}
We say that an $(m,r,r-1)$ Steiner system $\mc{S}$ is \emph{balanced} if $\lambda(\mc{S})=\lambda(\mc{S},\xi_{\mc{S}})$ (recall that $\xi_{\mc{S}}$ is defined in Section~\ref{sec:weghtedstability}; it is the uniform distribution on $V(\mc{S})$). It is easy to see that $m(\mf{B}(\mc{S}))=\pl{e}(m,r)$ when  $\mc{S}$ is balanced. The main result of this section, stated below, applies to all balanced Steiner systems. 
\begin{theorem}
 \label{thm:localstability} If $\mc{S}$ is a balanced $(m,r,r-1)$ Steiner system for some $m \geq r \geq 3$, then Forb$(\mc{T}_r)$ is $\mathfrak{B}(\mc{S})$-vertex locally stable.
 \end{theorem}

In all the following statements, $m\geq r \geq 3$ are fixed and  $\mc{S}$ is a balanced $(m,r,r-1)$ Steiner system. We denote $\mf{B}(\mc{S})$ simply by $\mf{B}$. The proof of Theorem~\ref{thm:localstability} uses three auxiliary  lemmas. The first ensures that if a large blowup \(\mc{B}\in \mf{B}\) has density close to the maximum possible (i.e. \(\pl{e}(m,r)\)), then the blowup partition is close being an equipartition. More formally, we say that the blowup \(\mc{B}\in \mf{B}\) with the blowup partition \(\mc{P}=\{P_1, P_2, \dots, P_m\}\) is \(\eps\)-\emph{balanced} for some \(0<\eps <1\), if  for each $ j \in [m]$, \[\left||P_j| - \frac{n}{m}\right|\leq \eps n.\]

\begin{lem}
\label{sizelemma} For every $\eps >0$ there exists $\delta >0$ and $n_0\in \mathbb{N}$ such that the following holds. If $\mc{B}\in\mf{B}$ with $\pl{v}(\mc{B})=n\geq n_0$ and $|\mc{B}|\geq \left(\pl{e}(m,r) - \delta \right) n^r$, then \(\mc{B}\) is \(\eps\)-balanced.
\end{lem}
\begin{proof}Let \(\mc{P}=\{P_1,P_2, \dots,P_m\}\) be the blowup partition of  \(\mc{B}\). Define a vector \(\pl{y}\) with $y_j = \frac{|P_j|}{n}$ for each $j \in [m]$. We have $\sum_{i=1}^m{y_j} = 1$ and \begin{equation}\label{lagr}\lambda(\mc{S},\pl{y}) = \frac{|\mc{B}|}{n^r} \geq \pl{e}(m,r)-\delta.
\end{equation}
Since $\mc{S}$ is balanced and $\lambda(\mc{S},\cdot)$ is a continuous function, for every $\eps>0$ there exists $\delta>0$ such that (\ref{lagr}) implies that $|y_j-1/m|\leq \eps,$
as desired.
\end{proof}

Before stating the second auxiliary lemma, we introduce additional definitions. Let $\mc{B} \in \mf{B}$ with the partition \(\mc{P}=\{P_1, P_2, \dots, P_m\}\) and $\mc{F} $  be an \(r\)-graph with $V(\mc{F})=V(\mc{B})$. We call the edges in $\mathcal{F}\setminus \mathcal{B}$ \emph{bad}, the edges in $\mc{B}\setminus \mc{F}$ \emph{missing} and, finally, the edges in $\mathcal{F} \cap \mc{B}$ \emph{good}. 

Given a collection of sets $\mc{X}=\{X_1,X_2,\ldots,X_k\}$ we say that a set $F$ is \emph{$\mc{X}$-transversal} if $|X_i \cap F| \leq 1$ for every $1 \leq i \leq k$.  We say that an $r$-graph $\mf{F}$ is \emph{$\mc{X}$-transversal} if every $F \in \mc{F}$ is $\mc{X}$-transversal.
Informally speaking, the next lemma tells us that if graphs  \(\mc{F}\) and \(\mc{B}\) are ``locally sufficiently close" and \(\mc{F}\) has density close to \(\pl{e}(m,r)\), then  \(\mc{F}\) must be \(\mc{P}\)-transversal. This result will be useful in the proof of Lemma~\ref{theorem2}, where working with bad edges we will be able to restrict our attention to transversal ones.

\begin{lem}\label{transversal}  There exist $\eps>0$ and $n_0\in \mathbb{N}$ such that the following holds. Let $\mc{F}$ be a $\mc{T}_r$-free $r$-graph with $\pl{v}(\mc{F}) = n\geq n_0$ vertices, $\mathcal{B}\in\mf{B}$ with $\pl{v}(\mc{B})=n$ and the blowup partition \(\mc{P}=\{P_1, P_2, \dots, P_m\}\). If $| L_{\mathcal{F}}(v)\triangle  L_{\mathcal{B}}(v)|\leq \eps n^{r-1}$ for every $v\in V(\mc{F})$, and $|\mc{F}|\geq \left(\pl{e}(m,r) - \eps \right) n^r$, then $\mc{F}$ is $\mc{P}$-transversal. Moreover, if $\mc{F}'$ is a $\mc{T}_r$-free $r$-graph such that  \(\mc{F} \subseteq \mc{F}'\), then $\mc{F}'$ is \(\mc{P}\)-transversal.
\end{lem}
\begin{proof} Clearly, it suffices to verify the last conclusion. Note that our choice of $n_0$ here (and in later proofs as well) is not explicit. We assume, for a contradiction, that there exists a non-transversal edge \(F \in \mc{F'}\), with \(v_1,v_2\in F\cap {P_j}\) for some  \(j\). We will show then that  \(\mc{F} \cup \{F\}\) contains a copy of \(\mc{T}_r\). We will find such a copy by showing the existence of an \((r-1)\)-tuple \(F'\in L(v_1)\cap L(v_2)\) that is disjoint from \(F\). Then, clearly \(F, F'\cup \{v_1\}\) and \(F'\cup\{v_2\}\) together will induce a \(\mc{T}_r\).

Let us specify the choice of constants used in the proof. Fix $\eps_{\ref{sizelemma}} :=\frac{1}{m+1}$. Let $\delta_{\ref{sizelemma}}$ be derived from Lemma~\ref{sizelemma} applied with $\eps = \eps_{\ref{sizelemma}}$.  We choose $0<\eps <1$ satisfying the following constraints
\begin{align}
\eps &< \pl{e}(m,r)\\
\eps&\left(1+\frac{1}{r}\right) \leq \delta_{\ref{sizelemma}} \label{epsilontrans:deltasize} \\
\eps &<\frac{1}{2}\pl{d}(m,r)\left(\frac{1}{m}-\eps_{\ref{sizelemma}}\right)^{r-1} \label{epsilontrans:epsilonsize}.
\end{align}

First, note that the links of both \(v_1\) and \(v_2\) have large size.  We have \begin{equation}\label{linksize}|L(v_i)| \geq d(m,r) \min_{i}{|P_i|^{r-1}}- \eps n^{r-1}.
\end{equation}
for $i = 1,2$.  But \(\mc{B}\) is an \(\eps_{\ref{sizelemma}}\)-balanced partition. Indeed, since
\[|\mc{F}\triangle\mc{B}| =\frac{1}{r}\sum_{i\in [n]}{| L_{\mc{F}}(i)\triangle L_{\mc{B}}(i)|} \leq \frac{1}{r}\eps n^r,\]
we have that  
\[|\mc{B}|\geq |\mc{F}| - \frac{1}{r}\eps n^r \geq \left(\pl{e}(m,r) - \eps\left(1+\frac{1}{r}\right)\right)n^r  \stackrel{(\ref{epsilontrans:deltasize})}{\geq} (\pl{e}(m,r)-\delta_{\ref{sizelemma}})n^r.\]
By Lemma~\ref{sizelemma}, applied to $\mc{B}$ with $\eps = \eps_{\ref{sizelemma}}$, we have
\[\left||P_j|-\frac{n}{m} \right|\leq \eps_{\ref{sizelemma}}n.\]
for each \(j\in[m]\). Thus, from (\ref{linksize}) it follows that
\[|L(v_i)| \geq \pl{d}(m,r)\left(\frac{1}{m} - \eps_{\ref{sizelemma}}\right)^{r-1}n^{r-1} - \eps n^{r-1}\]
for each $i\in \{1,2\}$.
Now we can show that the intersection of the links of \(v_1\) and \(v_2\) is large as well. Note that every $(r-1)$-tuple in $L(v_1)\triangle L(v_2)$ is either in a bad or in a missing edge with $v_1$ or $v_2$, but the total number of  such edges is bounded by the initial assumptions, hence 
\[|L(v_1)\triangle L(v_2)| \leq 2\eps n^{r-1}. \]
Thus,
\begin{align}\label{commonlink}|L(v_1)\cap L(v_2)| &= \frac{1}{2}\left(|L(v_1)| + |L(v_2)| - |L(v_1)\triangle L(v_2)| \right)\\
&\geq \pl{d}(m,r)\left(\frac{1}{m} - \eps_{\ref{sizelemma}}\right)^{r-1}n^{r-1} - 2\eps n^{r-1}\\
&>rn^{r-2},
\end{align}
where the last inequality is true for \(n\) sufficiently large. On the other hand, the number of $(r-1)$-tuples that do not contain both $v_1$ and $v_2$ and  have a common vertex with $F$ is bounded by $(r-2)n^{r-2}$. Hence, there exists an $(r-1)$-tuple \(F'\) in $L(v_1)\cap L(v_2)$ that is disjoint from $F$ and, as we discussed at the beginning of the proof, a contradiction follows.
\end{proof}

In the next lemma we show that for every $r$-graph $\mc{F} \in \brm{Forb}(\mc{T}_r)$  with sufficiently large minimum degree there exists a blowup $\mc{B}_0$ of $\mc{S}$ such that every vertex of $\mc{F}$ has ``similar'' neighborhoods in $\mc{F}$ and $\mc{B}_0$. The proof of this lemma contains the bulk of technical difficulties involved  in proving Theorem~\ref{thm:localstability}.

\begin{lem}
\label{theorem2}
For all integers $m\geq r\geq 3$ and $\eps>0$  there exists $\delta >0$ and $n_0\in \mathbb{N}$ such that the following holds. If $\mathcal{F}$ is a $\mc{T}_r $-free $r$-graph with $\pl{v}(\mc{F})= n \geq n_0$, $d_{\mf{B}}(\mathcal{F})\leq \delta n^r $, $|\mc{F}|\geq (\pl{e}(m,r)-\delta)n^r$ and for every $v\in V(\mc{F})$, $| L_{\mc{F}}(v)|\geq (\pl{d}(m,r) - \delta) n^{r-1}$, then there exists 
$\mc{B_0}\in\mf{B}$ with $v(\mc{B_0})=n$ such that for every $v\in V(\mc{F})$
\[| L_{\mathcal{F}}(v)\triangle  L_{\mathcal{\mc{B}_0}}(v)|\leq \eps n^{r-1}.\]
\end{lem}

\begin{proof}[Proof of Lemma~\ref{theorem2}:] 
Let $\eps_{\ref{transversal}}$ be chosen to satisfy Lemma~\ref{transversal}. We choose $$ 0<\delta\ll\eps_{\ref{sizelemma}} \ll \gamma\ll\min\{\eps_{\ref{transversal}},\eps\}$$ to satisfy the constraints appearing further in the proof.
Let $\delta_{\ref{sizelemma}}$ be chosen to satisfy Lemma~\ref{sizelemma} applied with $\eps=\eps_{\ref{sizelemma}}$. We assume that
$\delta \ll \delta_{\ref{sizelemma}}$.

Let $\mc{B}\in \mf{B}$ be such that $ |\mc{F}\triangle\mc{B}|= d_{\mf{B}}(\mc{F})$, and let $\mc{P}=\{P_1, P_2, \dots, P_m\}$ be the blowup partition of $\mc{B}$. Since $\delta< \delta_{\ref{sizelemma}}$ we have
 \[|\mc{F}|\geq (\pl{e}(m,r) - \delta)n^r \geq (\pl{e}(m,r) - \delta_{\ref{sizelemma}})n^r.\] Hence,  \(\mc{B}\) is \(\eps_{\ref{sizelemma}}\)-balanced by Lemma~\ref{sizelemma}.
 
Consider the set \[J:=\left\{v\in V(\mc{F}) | \left| L_{\mc{F}}(v)\triangle  L_{\mc{B}}(v)\right| > \gamma n^{r-1}\right\}.\] 
We have
\[|J|\gamma  n^{r-1}<\sum_{i\in [n]}{\left| L_{\mc{F}}(i)\triangle  L_{\mc{B}}(i)\right| } =r|\mc{F}\triangle \mc{B}|\leq \delta rn^r.\]
Let $\delta_1:=\delta r /\gamma$, then
$|J| \leq \delta_1 n$, by the above. Let $\mathcal{F}' := \mathcal{F}|_{V(\mc{F})\setminus J}$, $n'=\pl{v}(\mc{F}')$, $\mathcal{B}':= \mathcal{B}|_{V(\mc{F})\setminus J}$, $P_j':=P_j\setminus J$ for each $j\in[m]$, and $\mc{P}'=\{P'_1, P'_2, \dots, P'_m\}$. The graph \(\mc{F}'\) satisfies the assumptions of Lemma~\ref{transversal}. Indeed, for every \(v\in V(\mc{F}')\),
\[| L_{\mc{F}'}(v)\triangle  L_{\mc{B}'}(v)| \leq \gamma n^{r-1} \leq\eps_{\ref{transversal}} (1-\delta_1)^{r-1} n^{r-1}
\leq
\eps_{\ref{transversal}} (n')^{r-1}.\]
Similarly,
\begin{align*}|\mc{F}|\geq (\pl{e}(m,r) -\eps_{\ref{transversal}}) (n')^{r-1}.
\end{align*}
Thus both  $\mc{F}'$ and $\mc{F}$ are $\mc{P}$-transversal by Lemma~\ref{transversal}. Our next goal is to extend \(\mc{B}'\) to a blowup \(\mc{B}_0\) of $\mc{S}$ with $V(\mc{B}_0)=V(\mc{F})$, as follows. For each $u\in J$  we will find a unique index $j_{u} \in [m]$, such that $u$ ``behaves" as the vertices in the partition class $P_{j_u}'$, and add the vertex \(u\) to this partition class. By doing so for all vertices of \(J\), we will extend the partition \(\mc{P}'\), and since \(J\) has relatively small size, this operation will not increase the degrees of vertices in \(\mc{F}'\) drastically. So let us fix some \(u \in J\) and show that such an index \(j_u\) exists.

For $I \subseteq [m]$, let $$E_I(u):=\{F \in \mc{F} \:  | u\in F, \: |F \cap P'_i| =1 \: \mathrm{for\: every}\: i \in I  \}.$$
We construct an auxiliary   $(r-1)$-graph  $\mc{L}(u)$ with $V(\mc{L}(u))=[m]$ such that $I\in \mc{L}(u)$ if and only if $\left|E_I(u)\right|\geq \gamma n^{r-1}.$
We aim to show that there exists a unique $j_u \in [m]$ such that $\mc{L}(u)$ is isomorphic to the link graph of $j_u$ in $\mc{S}$. 

We start by proving that $\mc{L}(u)$ is at least as large as any of the link graphs \(L_{\mc{S}}(j)\), for \(j\in [m]\).
Denote by $E_J(u)$ the set of all the edges in $\mc{F}$ that contain $u$ and at least one other vertex from $J$. Clearly,
\(|E_J(u)|\leq |J|n^{r-2} \leq \delta_1 n^{r-1}.\) Therefore, 
\begin{align*}
(\pl{d}(m,r)-\delta)n^{r-1}&\leq | L_{\mc{F}}(u)|
\leq  |E_J(u)|+\sum_{I\in \mc{L}(u)}{|E_I(u)|} + \sum_{I\notin \mc{L}(u)}{|E_I(u)|}\\
&\leq \delta_1 n^{r-1}+ |\mc{L}(u)| \left(\frac{1}{m}+\eps_{\ref{sizelemma}}\right)^{r-1}n^{r-1}+ \gamma {m \choose r-1} n^{r-1}.
\end{align*}
It follows that 
\begin{align*}
|\mc{L}(u)|&\geq \frac{\pl{d}(m,r)m^{r-1}}{(1+\eps_{\ref{sizelemma}}m)^{r-1}} - \frac{(\delta+\delta_1 +\gamma/(r-1)!)m^{r-1}}{(1+\eps_{\ref{sizelemma}}m)^{r-1}}\\ &> \pl{d}(m,r)m^{r-1}-1,
\end{align*}
where the last inequality holds, as long as $\eps_{\ref{sizelemma}},\delta, \delta_1$ and $\gamma$ are sufficiently small compared to $1/m^r$.
It follows that $|\mc{L}(u)|\geq \pl{d}(m,r)m^{r-1} = |L_{\mc{S}}(j)|$ for any \(j\in[m]\). 

Next, we find \(j_u\) such that \(\mc{L}(u)\subseteq L_{\mc{S}}(j_u)\). For every $j \in [m]$ consider
 \[ L_{j}(u) := \{v \in P_j' : |L_{\mc{F}}(\{u,v\})|\geq \gamma n^{r-2} \},\]
that is, \( L_{j}(u)\) is the set of vertices in the partition class \(P_j'\) which are in relatively many edges with \(u\).  Let \(K=\{j : \left| L_{j}(u)\right|< \gamma n\}\). We want to show that $|K|=1$, from which it will follow that \(u\) essentially behaves as the vertices of the partition class corresponding to this unique index in \(K\). 

First, let us prove that  \(K \neq \emptyset\). Fix $I\in \mc{L}(u)$. As $\mc{S}$ is a Steiner system, there exists unique $j$ such that  $I\cup \{j\}\in \mc{S}$. We claim that $j \in  K$. Assume not, and further assume, without loss of generality, that \(I=\{1,2,\dots, r-1\}\). Then there  exists $\{v_1,v_2, \dots, v_{r-1}\}\in E_I(u)$ and $v_r\in  L_{j}(u)$, such that $\{v_1,v_2, \dots, v_{r-1},v_r\}\in\mc{F}$. Otherwise, for every $F\in E_I(u)$ and every \(v\in L_j(u)\), $\left(F\setminus\{u\}\right)\cup \{v\}$ is a missing edge. Hence,
\[|\mc{F}\triangle \mc{B}| \geq |E_I(u)| | L_j(u)| \geq \gamma n^{r-1}\cdot \gamma n >\delta n^r, \]
a contradiction. 

Let $v_1,v_2, \dots, v_{r-1}, v_r$ be as above. Since $\mc{F}$ is $\mc{T}_r $-free, every edge in $\mc{F}$ that contains both $u$ and $v_r$, must also contain a vertex among $\{v_1,v_2, \dots, v_{r-1}\}$. Therefore, we must have
\(|L(\{u,v_r\})| \leq (r-1)n^{r-3}\), while, by definition of \( L_j(u)\), \(|L(\{u,v_r\})|\geq\gamma n^{r-2}\), 
yielding a contradiction when $n$ is large enough. Thus $K \neq \emptyset$. Note that if we prove that $K = \{j_u\}$ for some index \(j_u\), then since $|\mc{L}(u)|\geq m^{r-1}\pl{d}(m,r)=|L_{\mc{S}}(j_u)|$, it will follow that $\mc{L}(u)=L_{\mc{S}}(j_u)$.

\begin{claim}\label{claim:K=1} $|K|= 1$.
\end{claim}
\begin{proof}
Let \(k:=|K|\), we have already shown that \(k\geq 1\). Suppose for a contradiction that \(k\geq 2\). Let \(A\) be a $\mc{P}'$-transversal $(r-2)$-tuple. We want to show that
\begin{equation}\label{linkofiandI}|L(A\cup \{u\})|\leq\left (\frac{1}{m} + \eps_{\ref{sizelemma}} + \gamma(m-1)\right) n.
\end{equation}
Suppose that there exist $j_1\neq j_2$ such that $|L(A\cup \{u\})\cap P_{j_1}| \geq \gamma n$ and $|L(A\cup \{u\})\cap P_{j_2}| \geq \gamma n$. Since $\mc{F}$ is $\mc{T}_r $-free, for every $v_1\in L(A\cup \{u\})\cap P_{j_1}$ and $v_2\in L(A\cup \{u\})\cap P_{j_2}$, we must have
$$|L(\{v_1,v_2\})|\leq (r-1)n^{r-3}.$$ 
It follows that 
\begin{align*}
|\mc{F}\triangle\mc{B}|&\geq \gamma^2 n^2\left(\left(\frac{1}{m}-\eps_{\ref{sizelemma}}\right)^{r-2} n^{r-2} - (r-1){n \choose r-3}\right)\\
&{\geq} \frac{1}{2}\gamma^2 \left(\frac{1}{m}-\eps_{\ref{sizelemma}}\right)^{r-2} n^{r}>\delta n^r,
\end{align*}
which is a contradiction. Thus, no such $j_1$ and $j_2$ exist, and (\ref{linkofiandI}) follows. 

Using (\ref{linkofiandI}), we  obtain an upper bound on \(E_I(u)\), for every \((r-2)\)-tuple \(I\subseteq [m]\). Without loss of generality, suppose $I=\{1,2, \dots, {r-2}\}$. We apply  (\ref{linkofiandI}) to every \(A\in [n]^{r-2}\) which is \(I\)-transversal in \(\mc{F}\) (i.e. $|A\cap P_i| =1$ for every $i\in I$). As
$$\prod_{j=1}^{r-2} |P_{j}| \leq \left(\frac{n}{m}+\eps_{\ref{sizelemma}}n\right)^{r-2},$$
we derive
\begin{equation}\label{linkofiandI2}|E_I(u)|\leq\left(\frac{1}{m}+\eps_{\ref{sizelemma}}\right)^{r-2} \left(\frac{1}{m} + \eps_{\ref{sizelemma}} + \gamma(m-1)\right)n^{r-1}.
\end{equation}

And finally, we are ready to derive an upper bound on the size of \( L_{\mc{F}}(u)\), which will contradict the initial assumption \(| L_{\mc{F}}(u)| \geq (\pl{d}(m,r)-\delta) n^{r-1}\):
\begin{align*}
| L_{\mc{F}}(u)|&\leq  |E_J(u)|+ \sum_{I\subseteq[m],|I|=r-1 \atop{I\cap K = \emptyset}}{|E_I(u)|} + \sum_{I\subseteq[m],|I|=r-1 \atop{I\cap K \neq \emptyset}}{|E_I(u)|} 
\\
&  \leq |J|n^{r-2}+ \frac{1}{r-1}\sum_{I\subseteq[m],|I|=r-2 \atop{I\cap K = \emptyset}}{|E_I(u)|} \\& + \sum_{j\in K}\left(| L_{j}(u)| n^{r-2} +(n-| L_{j}(u)|)\gamma n^{r-2}\right)
\\
&\stackrel{(\ref{linkofiandI2})}\leq \frac{1}{r-1}\binom{m-s}{r-2}\left(\frac{1}{m}+\eps_{\ref{sizelemma}}\right)^{r-2} \left(\frac{1}{m} + \eps_{\ref{sizelemma}} + \gamma(m-1)\right)n^{r-1}
\\&+\delta_1n^{r-1}+ 2\gamma m n^{r-1}\\
&\leq \left( \left(\frac{\binom{m-2}{r-1}}{m^{r-1}} +\gamma\right)+2\gamma m + \delta_1\right) n^{r-1}\\
&< (\pl{d}(m,r)-\delta) n^{r-1},
\end{align*}
a contradiction. Thus, \(k = 1\).
\end{proof}

As discussed above, Claim~\ref{claim:K=1} implies that for every \(u\in J\) there exists unique \(j_u\) such that $\mc{L}(u)=L_{\mc{S}}(j_u)$. We  extend the blowup \(\mc{B}'\) as we discussed earlier. For every \(j\in [m]\), define $$P_{j}^{0}:=P_j'\cup \{u\in J \: |\:  j_u = j\}.$$ Let $\mc{B}_0 \supseteq \mc{B}'$ be the blowup of $\mc{S}$ with the blowup partition \(\mc{P}_0\).  
\begin{claim}\label{B0claim}For every \(v\in V(\mc{F})\), 
\[| L_{\mc{B}_{0}}{(v)}\triangle  L_{\mc{F}}(v)|\leq \eps n^{r-1}.\]
\end{claim}
\begin{proof}
For each $v\in V(\mc{F})\setminus J$, we have 
\begin{align*}| L_{\mc{B}_{0}}{(v)}\triangle  L_{\mc{F}}(v)|&\leq | L_{\mc{B}'}{(v)}\triangle  L_{\mc{F'}}(v)| + |J|n^{r-2} \\
&\leq \gamma n^{r-1} + \delta_1 n^{r-1} \leq \eps n^{r-1}.
\end{align*}
We now consider $v\in J$. Since $\mc{F}$ is $\mc{P}'$-transversal, it follows that for every $F \in  L_{\mc{F} \setminus \mc{B}_0}(v)$, either $F \cap J \neq \emptyset$, or there exists $I \not \in \mc{L}(v)$ such that $F \in L_I(v)$. Thus,
\begin{equation}\label{eq:theorem2} | L_{\mc{F} \setminus \mc{B}_0}(v) |
\leq \delta_1 n^{r-1} +  \left({m \choose r-1} - |\mc{L}(v)|\right)\gamma n^{r-1}< \frac{\eps}{8}n^{r-1}.
\end{equation}
Finally,
\begin{align*} |& L_{\mc{F}}(v)\triangle  L_{\mc{B}_0}(v)| = 
2|  L_{\mc{F} \setminus\mc{B}_0}(v) | + | L_{\mc{B_0}}(v)| - | L_{\mc{F}}(v)| \\ &\stackrel{(\ref{eq:theorem2})}\leq \frac{\eps}{2}n^{r-1} +\pl{d}(m,r) \left(\frac{1}{m}+\eps_{\ref{sizelemma}}+\delta_1\right)^{r-1}n^{r-1}-(\pl{d}(m,r)-\delta)n^{r-1}\\&\leq \eps n^{r-1},
\end{align*}
as desired.
\end{proof}
By Claim~\ref{B0claim} the blowup \(\mc{B}_0\) satisfies the conclusion of the lemma, thus finishing the proof.
\end{proof}
\vskip 10pt
We are now ready for the proof of Theorem~\ref{thm:localstability}.

\begin{proof}[Proof of Theorem~\ref{thm:localstability}.] Our goal is to show that there exist $\eps,\alpha,n_0>0$ such that the following holds. If $\mc{F}\in \brm{Forb}(\mc{T}_r)$ with $\pl{v}(\mc{F})= [n]$, $n \geq n_0$ such that  $d_{\mf{B}}(\mc{F})\leq \eps n^r$, and $| L_{\mc{F}}(v)|\geq (\pl{d}(m,r)- \eps)n^{r-1}$ for every $v\in V(\mc{F})$,  then
\begin{equation}
\label{eq:localstab}
|\mc{F}|\leq m(\mf{B},n) - \alpha d_{\mf{B}}(\mc{F}).
\end{equation}
In fact, we show that one can take $\alpha = \frac{1}{2}$. Now we specify dependencies between constants used further in the proof. Let $\eps_{\ref{transversal}}$ be taken to satisfy Lemma~\ref{transversal}. Define $\eps_{\ref{sizelemma}}:=\frac{1}{4m}$. Let $\delta_{\ref{sizelemma}}$ be taken to satisfy  Lemma~\ref{sizelemma}  applied with $\eps=\eps_{\ref{sizelemma}}$. We choose $0\ll \eps \ll \eps_{\ref{theorem2}}\ll \min\{\delta_{\ref{sizelemma}},\eps_{\ref{transversal}}\}$ to satisfy the inequalities appearing in the proof. In particular, we will use $\eps < \delta_{\ref{theorem2}}/2$, where $\delta_{\ref{theorem2}}$ is chosen to satisfy Lemma~\ref{theorem2} applied with $\eps_{\ref{theorem2}}$.

We can assume that 
\[|\mathcal{F}|\geq (\pl{e}(m,r)-2\eps)n^r \geq (\pl{e}(m,r)-\delta_{\ref{theorem2}})n^r,\] 
since otherwise the result follows directly with $\alpha =1$.  By Lemma~\ref{theorem2} there exists $\mc{B}\in \mf{B}$ with $V(\mc{B}) =V(\mc{F})$ such that 
$$| L_{\mc{F}}(v)\triangle  L_{\mc{B}}(v)|\leq \eps_{\ref{theorem2}}n^{r-1}$$ for every $v\in V(\mc{F}).$

Recall the definitions of missing and bad edges at the beginning of this section. Generalizing these notions, we introduce the following notation. For every $I\subset V(\mc{F})$ with $0\leq |I|\leq r$, we denote
\[A(I) :=\{F\in \mc{B}\setminus \mathcal{F}| I\subseteq F\},\]
\[B(I) :=\{F\in \mathcal{F}\setminus \mc{B}| I\subseteq F\},\]
$a(I) := |A(I)|$ and $b(I) :=|B(I)|$. So \(a(I)\) and \(b(I)\) respectively denote the number of missing and bad edges that the tuple \(I\) is in. We have $\mc{F}\triangle \mc{B} = A(\emptyset) \cup B(\emptyset)$ and \(|\mc{F}\triangle \mc{B}| = a(\emptyset) + b(\emptyset)\).  It is easy to see that for every  $I$, such that $0\leq |I|\leq r-1$, the following inequalities hold
\begin{equation}
\label{a}
\sum_{j\notin I}{a(I\cup \{j\})}\geq a(I)\geq \frac{1}{r}\sum_{j\notin I}{a(I\cup \{j\})},
\end{equation}
\begin{equation}
\label{b}
\sum_{j\notin I}{b(I\cup \{j\})}\geq b(I)\geq \frac{1}{r}\sum_{j\notin I}{b(I\cup \{j\})}.
\end{equation}
It is not hard to see that to derive the inequality (\ref{eq:localstab}) it suffices to show that $a(\emptyset)\geq 3 b(\emptyset)$. Let us assume for a contradiction that $b(\emptyset)>\frac{1}{3}a(\emptyset)$. Our next claim shows that we can bound the number of bad edges that contain some \(i\)-tuple from above by the proportion of the missing edges that contain any of its \((i-1)\)-subtuples.

\begin{claim}\label{amplificationlem}There exists $c>0$ such that for every $I \subseteq V(\mc{F}), 1 \leq |I| \leq r$, and every $I'\subset I$ with  $|I'|=|I|-1$, we have $a(I')\geq c b(I)n$.
\end{claim}
\begin{proof} We proceed by induction on $r-|I|$. We prove that for each $1 \leq i \leq r$, and every $I\subseteq [n]$ with $|I|=i$ there exists $c_i>0$ such that for all $I'\subset I$ and  $|I'|=i-1$, we have $a(I')\geq c_ib(I)n$. This clearly implies the claim.

We start the base case: $|I|=r$ and we assume that $I$ is a bad edge, as otherwise the statement is trivial. Let $\mc{P}=\{P_1, P_2, \dots, P_m\}$ be the blowup partition of $\mathcal{B}$. By our assumptions,
$|\mathcal{F}|\geq (\pl{e}(m,r)-\eps_{\ref{transversal}})n^r$
and $| L_{\mc{F}}(v) \triangle  L_{\mc{B}}(v)|\leq \eps_{\ref{theorem2}} n^r \leq \eps_{\ref{transversal}}n^r$ for every \(v\in V(\mc{F})\). Thus by Lemma~\ref{transversal} all bad edges in $\mc{F}$ are $\mc{P}$-transversal. 

Without loss of generality, assume $I=\{v_1,v_2,\dots,v_r\}$,
where $v_j\in P_j$, and $I'=\{v_1,v_2,\dots,v_{r-1}\}$. Since \(I\) is a bad edge, it means that $\{1,2,\dots, r\}\notin \mc{S}$ which implies that  $\{1,2,\dots, r-1, k\}\in \mc{S}$ for some \(k\neq r\). Without loss of generality, we assume $k=r+1$. 

Let $N:=L(I') \cap { P_{r+1}}$. For every $u\in N$, we have $$a(\{u,v_r\}) \geq (\min_{i}|P_i|)^{r-2}-|L(\{u,v_r\})|.$$ However, every edge that covers $u$ and $v_r$, must have a non-empty intersection
with $\{v_1,v_2,\dots,v_{r-1}\}$, as $\mc{F}$ is $\mc{T}_r $-free, therefore
\[|L(\{u,v_r\})|\leq (r-1)n^{r-3}.\]   
As $|\mathcal{F}| \geq (\pl{e}(m,r)-2\eps)n^r \geq (\pl{e}(m,r)-\delta_{\ref{sizelemma}})n^r$ the blowup \(\mc{B}\) is \(\epsilon_{\ref{sizelemma}}\)-balanced by Lemma~\ref{sizelemma}.  Therefore
\[a(\{v_r\}) \geq |N|\left(\left(\frac{n}{m}-\eps_{\ref{sizelemma}} n\right)^{r-2} - (r-1)n^{r-3} \right).\] 
But \( a(\{v_r\})\leq \eps_{\ref{theorem2}} n^{r-1}\) and we have
\[|N|\leq \frac{2\eps_{\ref{theorem2}}}{\left(\frac{1}{m}-\eps_{\ref{sizelemma}}\right)^{r-2}}n = 2\eps_{\ref{theorem2}}\left(\frac{4m}{3}\right)^{r-2}n \leq \frac{n}{2m}, \]
for sufficiently large $n$.
The latter directly implies that \(a(I')\geq |P_{r+1}\setminus N| \geq \frac{n}{4m}\), thus concluding the proof of the base case with $c_r= \frac{1}{4m}$.

We now turn to the induction step. For every $I'\subset I$ with $|I'|=|I|-1$ we have 
\begin{align*}
ra(I') &\stackrel{(\ref{a})}{\geq} \sum_{I'\subset J, \atop{ |J| = i}}{a(J)}
&\geq\sum_{I'\subset J, J\neq I \atop{ |J| = i}}{c_{i+1}b(J\cup I)n}
&\stackrel{(\ref{b})}{\geq}c_{i+1}b(I)n,
\end{align*}
where the second inequality follows from th induction hypothesis. Thus $a(I')\geq c_{i} b(I) n$, where $ c_i:=\frac{c_{i+1}}{r}>0$, as desired.
\end{proof}

Let $c$ be as in Claim~\ref{amplificationlem}.  Then $a(\emptyset)\geq cb(v)n$ for every $v\in V(\mc{F})$. Direct averaging shows that for every $I \subseteq V(\mc{F})$ with $0\leq |I|\leq r-1$ and every $c'>0$ such that $b(I) > c' a(I)$, there exists $v\notin I$ such that 
$b(I\cup \{v\})> c'a(I\cup \{v\})$. Therefore, since $b(\emptyset)>\frac{1}{3}a(\emptyset)$, there exists $v_1\in V(\mc{F})$ such that $b(\{v_1\})>\frac{1}{3}a(\{v_1\})$.
Similarly, $a(\{v_1\})\geq cb(\{v_1,v\})$ for every $v\in V(\mc{F}) \setminus \{v_1\}$, and there exists $v_2\in V(\mc{F}) \setminus \{v_1\}$, such that $b(\{v_1,v_2\})>\frac{1}{3}a(\{v_1,v_2\})$. Applying this argument iteratively, we get the following series of inequalities:
\begin{align*}
a(\emptyset)&\geq cb(\{v_1\})n >\frac{c}{3}a(\{v_1\})n\geq\frac{c^2}{3}b(\{v_1,v_2\})n^2>\frac{c^2}{9}a(\{v_1,v_2\})n^2\geq \dots \\
&>\frac{c^{r-1}}{3^{r-1}}a(\{v_1,v_2,\dots,v_{r-1}\})n^{r-1}
\geq \frac{c^r}{3^{r-1}}b(\{v_1,v_2,\dots,v_r\})n^r \\& > \frac{c^r}{3^{r}}a(\{v_1,v_2,\dots,v_r\})n^r. 
\end{align*}
In particular,  $b(\{v_1,v_2,\dots,v_r\})>0$, i.e.   $b(\{v_1,v_2,\dots,i_r\})=1$. Thus, 
\[a(\emptyset) > \frac{c^r}{3^{r-1}}n^r \geq\frac{\eps_{\ref{theorem2}}}{r}n^r \geq |\mc{F}\triangle\mc{B}|,\]
a contradiction. 
\end{proof}
\end{section}

\begin{section}{Proof of Theorem~\ref{thm:maintheorem}}\label{sec:finale}
In this section we combine all of the preceding results to prove Theorem~\ref{thm:maintheorem}. In fact, we prove a stronger theorem which directly implies Theorem~\ref{thm:maintheorem}.  We adopt the following notation for the rest of the section:  $\hat{\mf{F}}=\brm{Forb}(\Sigma_r)$ and \(\mf{F}^*:=\{\mc{F}\in \hat{\mf{F}} \: | \: \mc{F} \text{ covers pairs }\}\). We say that an $r$-graph $\mc{F}$ is \emph{uniquely dense} (\emph{around} \(\Sigma_r\)) if $\lambda(\mc{F},\xi_{\mc{F}}) \geq \lambda(\mc{F}^*,\mu)$ for every $\mc{F}^* \in \mf{F}^*$, $\mu \in \mc{M}(\mc{F}^*)$ and, further, the equality holds only when  $\mc{F}^*$ is isomorphic to $\mc{F}$ and $\mu = \xi_{\mc{F}^*}$. 

\begin{theorem}\label{thm:general}
If \(\mc{S}\) is a uniquely dense $(m,r,r-1)$ Steiner system for some \(m\geq r\geq 3\), then $\brm{Forb}(\mc{T}_r)$ is $\mf{B}(\mc{S})$-stable. 
\end{theorem}

Let \(\mc{S}_5\) and \(\mc{S}_6\) denote the unique $(11,5,4)$ and  $(12, 6,5)$ Steiner systems respectively.  The following result of Frankl and F\"{u}redi allows us to immediately derive Theorem~\ref{thm:maintheorem} from Theorem~\ref{thm:general}.

\begin{theorem}[P. Frankl, Z. F\"{u}redi, \cite{franklfuredi}] \label{thm:FranklFuredi} $\mc{S}_5$ and $\mc{S}_6$ are uniquely dense.
\end{theorem}

It remains to prove Theorem~\ref{thm:general}.

\begin{proof}[Proof of Theorem~\ref{thm:general}.] Let $\mf{F}:=\brm{Forb}(\mc{T}_r)$ and $\mf{B}:=\mf{B}(\mc{S})$. Note that a uniquely dense Steiner system is, in particular,  balanced.  Therefore, by Theorem~\ref{thm:localstability},  $\mf{F}$ is $\mf{B}$-vertex locally stable. Clearly $\mf{B}$ is clonable, thus from Theorem~\ref{thm:narrowlocaltolocal} it follows that  $\mf{F}$ is $\mf{B}$-locally stable. We derive $\mf{B}$-stability of  $\mf{F}$  from $\mf{B}$-stability of $\hat{\mf{F}}$ (which we will prove in a second) and the $\mf{B}$-local stability of \(\mf{F}\), combined with the following application of the Hypergraph Removal Lemma.

\begin{theorem}\label{FisFhatstable}For every $\eps>0$ there exists $n_0\in \mathbb{N}$ such that for every $\mc{F} \in \mf{F}$ with $\brm{v}(\mc{F})=n \geq n_0$ there exists $\hat{\mc{F}} \subseteq \mc{F}$, $\hat{\mc{F}} \in \hat{\mf{F}}$ such that $$|\hat{\mc{F}}| \geq |\mc{F}|-\eps n^r.$$
\end{theorem}
We will omit the proof of Theorem~\ref{FisFhatstable}, an interested reader can find it in ~\cite{pikhurko}.  

It is easy to see that $\mf{F}^*$  is thin. Since \(\mc{S}\) is uniquely dense, we have
$$\{\mc{S}\}=\{\mc{F} \in \mf{F}^{*} \: | \: \lambda(\mc{F},\mu)= \lambda(\mf{F}^{*})\; \mathrm{for\; some} \; \mu \in \mc{M}(\mc{F})\}=:\mf{F}^{**}.$$

Thus Theorem~\ref{thm:compactness} implies that \(\mf{F}^*\) is \(\mf{B}\)-weakly weight stable, and therefore, by Theorem~\ref{thm:symmetrization} the family $\hat{\mf{F}}$ is $\mf{B}$-stable. Note that here we are using the fact that both $\hat{\mf{F}}$ and $\mf{B}$ are clonable families (unlike $\mf{F}$ which is not clonable).

Let $\alpha,\eps>0$ be such that $\mf{F}$ is $(\mf{B},\alpha,\eps)$-locally stable and $\hat{\mf{F}}$ is $(\mf{B},\alpha)$-stable. We claim that $\mf{F}$ is $(\mf{B},\alpha/2)$-stable. Indeed, consider $\mc{F} \in \mf{F}$ with $\brm{v}(\mc{F})=n$. We want to show that
\begin{equation}\label{eq:localstabilityend}
|\mc{F}|\leq m(\mf{B},n) - \frac{\alpha}{2} d_{\mf{B}}(\mc{F}),
\end{equation}
if $n$ is sufficiently large. If  $d_{\mf{B}}(\mc{F}) \leq \eps n^r$ then (\ref{eq:localstabilityend})  holds, as $\mf{F}$ is $(\mf{B},\alpha,\eps)$-locally stable, and so we can assume that $d_{\mf{B}}(\mc{F}) > \eps n^r$. By  Theorem~\ref{FisFhatstable} there exists $\hat{\mc{F}} \subseteq \mc{F}$ such that $|\hat{\mc{F}}| \geq |\mc{F}|-\eps' n^r$, where we choose $\eps':=\frac{\alpha}{2(\alpha+1)}\eps$. As $\hat{\mf{F}}$ is $(\mf{B},\alpha)$-stable, we have 
\begin{align*}
|\mc{F}| &\leq |\hat{\mc{F}}| + \eps' n^r \leq m(\mf{B},n) - \alpha d_{\mf{B}}(\hat{\mc{F}})+ \eps' n^r \\&\leq m(\mf{B},n) - \alpha( d_{\mf{B}}(\mc{F}) - \eps'n^r)+ \eps' n^r \\&= m(\mf{B},n) - \frac{\alpha}{2} d_{\mf{B}}(\mc{F}) + \left((\alpha+1)\eps'n^r-\frac{\alpha}{2} d_{\mf{B}}(\mc{F})\right)\\&\leq m(\mf{B},n) - \frac{\alpha}{2} d_{\mf{B}}(\mc{F}), 
\end{align*}
where the last inequality holds by the choice of $\eps'$. This concludes the proof of the theorem.
\end{proof}

\end{section}

\bibliographystyle{amsplain}
\bibliography{lib}
\end{document}